\documentclass[final,leqno]{siamltex}

\usepackage{epic,eepic}
\usepackage{multirow}
\usepackage{tikz}
\usepackage{amsfonts,amssymb} 
\usepackage{amsmath} 
\usepackage{color} 
\usepackage{graphicx} 
\usepackage{epstopdf}
\usepackage{graphicx,ifthen,latexsym,mathrsfs,multicol}
\usepackage[bf,SL,BF]{subfigure} 
\usepackage{fancyhdr} 
\usepackage{CJK} 
\usepackage{caption} 
\usepackage{wrapfig} 
\usepackage{cases} 
\usepackage{bm}
\usepackage{algorithm}
\usepackage{algorithmic}
\newtheorem{remark}{Remark}[section] 

\usepackage{booktabs}
\usepackage{mathtools}

\title{Fast and High-order Accuracy Numerical Methods for  Time-Dependent Nonlocal Problems in $\mathbb{R}^2$
\thanks{This work was supported by NSFC 11601206, 11471150 and the Fundamental Research Funds for the Central Universities under Grant No. lzujbky-2019-80. Research
supported in part by the HKRGC GRF 12306616, 12200317, 12300218 and 12300519, and
HKU Grant 104005583.}}

\author{Rongjun Cao\thanks{ School of Mathematics and Statistics, Gansu Key Laboratory of Applied Mathematics and Complex Systems,
 Lanzhou University, Lanzhou 730000, P.R. China  (Email: caorj18@lzu.edu.cn)}
\and Minghua Chen\thanks{
Corresponding author. School of Mathematics and Statistics, Gansu Key Laboratory of Applied Mathematics and Complex Systems,
 Lanzhou University, Lanzhou 730000, P.R. China  (Email: chenmh@lzu.edu.cn) }
 \and Michael K. Ng\thanks{
Department of Mathematics, The University of Hong Kong, Pokfulam, Hong Kong  (Email: mng@maths.hku.hk)}
  \and Yu-Jiang Wu\thanks{
School of Mathematics and Statistics, Gansu Key Laboratory of Applied Mathematics and Complex Systems, Lanzhou University, Lanzhou 730000, P.R. China  (Email: myjaw@lzu.edu.cn)}
 }

\begin{document}

\maketitle

\begin{abstract}
In this paper, we study the Crank-Nicolson method for temporal dimension
and the piecewise quadratic polynomial collocation method for spatial dimensions
of time-dependent nonlocal problems. The new theoretical results of such discretization
are that the proposed numerical method is
unconditionally stable and its global truncation error is of
$\mathcal{O}\left(\tau^2+h^{4-\gamma}\right)$ with $0<\gamma<1$, where $\tau$ and $h$ are
the discretization sizes in the temporal and spatial dimensions respectively.
Also we develop the conjugate gradient squared method to solving
the resulting discretized nonsymmetric and indefinite systems arising from time-dependent
nonlocal problems including  two-dimensional cases. By using additive and multiplicative Cauchy kernels in non-local problems,
structured coefficient matrix-vector multiplication can be performed efficiently in the conjugate gradient squared
iteration. Numerical examples are given to illustrate our theoretical results and demonstrate
that the computational cost of the proposed method is of $O(M \log M)$ operations where $M$
is the number of collocation points.
\end{abstract}



\pagestyle{myheadings}
\thispagestyle{plain}
\markboth{R. J. CAO, M. H. CHEN, M. K. NG AND Y. J. WU}{NUMERICAL METHODS FOR
TIME-DEPENDENT NONLOCAL PROBLEMS}

\section{Introduction}\label{sec:1}

In this paper, we study an error estimate and develop fast conjugate gradient squares method of the piecewise quadratic polynomial
collocation (PQC) for time-dependent nonlocal problems, whose  prototype
is  \cite{Andreu:10,Bates:06,Du:12,Silling:00}
\begin{gather*}
u_t(x,t)+\int_\Omega J(|x-y|)\left[ u(x,t)-u(y,t) \right]dy  =f(x,t),   \quad (x,t) \in \Omega \times(0,T]. \tag{$*$}
\end{gather*}
Here $J(x)$ is a radial probability density with a nonnegative  symmetric dispersal kernel,
with nonhomogeneous  Dirichlet  boundary conditions  and initial condition $u(x,0)=u_0(x)$.
There are a lot of scientific phenomena that can be described by model ($*$) in various applications, for example, in
materials science, biology, particle systems, image processing, coagulation models, mathematical finance, see \cite{Andreu:10,Bates:06}
for detailed discussion.
The well-posedness (existence and uniqueness) of the model ($*$) can be found in the monograph \cite[p.\,46]{Andreu:10}.
We notice that there are many different choices to prescribe $J(x)$ for nonlocal problems ($*$), e.g., the constant kernel, fractional Laplacian kernel
or commonly used kernel \cite{Andreu:10,CD:17,Du:12,WT:12,ZhangWJ:18}
\begin{equation*}
  J(x)\sim\frac{1}{x^{1+2s}},~~s\in [-0.5,1).
\end{equation*}
In this paper,  we mainly focus on the case $s\in (-0.5,0)$, the other cases  can be similarly studied.
Then the  nonlocal model ($*$)  reduces to the following nonlocal diffusion problem
\begin{equation}\label{1.1}
\frac{\partial u(x,t)}{\partial t}+\int^b_a \frac{u(x)-u(y)}{|x-y|^\gamma}dy  =f(x,t),   \quad (x,t) \in (a,b)\times(0,T],
\quad 0< \gamma <1.
\end{equation}

To seek the numerical solution of time-dependent nonlocal problems, or specifically \eqref{1.1}, we employ the piecewise quadratic polynomial collocation method to approximate the following
weakly singular integral
\begin{equation*}
 I(a,b,x) =\int^b_a \frac{u(y)}{|x-y|^\gamma}dy,  \quad x \in (a,b), \quad 0< \gamma <1,
\end{equation*}
in the discretization for non-local problems in (1.1).
Note that the local truncation error with $\mathcal{O}\left(h^3\right)$ convergence was established in \cite{Aikinson:67}, where
$h$ is the discretization size in the spatial dimension.
The quasi-optimal error estimate with $\mathcal{O}\left(h^{4-\gamma}\right)$ convergence was provided in \cite{Hoog:73} or \cite[p.\,125]{Aikinson:09}.
By using the techniques of hypersingular integral \cite{GFTJZ:18,LiSun:10,WL:05},
researchers provided  an optimal error
$\mathcal{O}\left(h^4\eta_i^{-\gamma}\right)$, $\eta_i=\min\left\{x_i-a,b-x_i\right\}$ for the weakly singular integral \cite{CQSW:19}.
Numerical methods for the steady-state version of \eqref{1.1} have been proposed and studied in the literature.
For example, the second-order convergence results are provided in
\cite{CES:19,WT:12} by using the finite element method with piecewise linear polynomial basis.
Recently, numerical results for the steady-state version of \eqref{1.1} with $\gamma=1$
was shown that the convergence rate is close to 1.5 by the PLC method \cite{TWW:13}.
There is still no theoretical convergence results for the PLC method.
In \cite{CQSW:19}, Chen et al. showed the optimal first order and third-order convergence rates
for the PLC and the PQC methods respectively.
To the best of our knowledge, there is only a few study for time-dependent nonlocal problems.
In \cite{DHZZ:18}, Du et al. studied
the two-dimensional nonlocal wave equation on unbounded domains and its numerical solution based on quadrature scheme.

The main aim of this paper is to study the Crank-Nicolson method for temporal dimension
and the piecewise quadratic polynomial collocation method for spatial dimensions
of time-dependent nonlocal problems in (1.1).
The new theoretical results of such discretization are that the proposed numerical method is
unconditionally stable and its global truncation error is of
$\mathcal{O}\left(\tau^2+h^{4-\gamma}\right)$ with $0<\gamma<1$, where $\tau$ and $h$ are
the discretization sizes in the temporal and spatial dimensions respectively.
Also we employ the conjugate gradient squared method \cite{Saad:03,Sonneveld:89}
to solving the resulting discretized nonsymmetric and indefinite systems arising from time-dependent
nonlocal problems. By using additive and multiplicative Cauchy kernels in non-local problems,
structured coefficient matrix-vector multiplication \cite{Chen:0013,PNW:16,Pang:12,Wang:12,ChanNg:96,DuW:15}
can be performed efficiently in the conjugate gradient squared
iteration. Numerical examples are given to illustrate our theoretical results and demonstrate
that the computational cost of the proposed method is of $O( M \log M)$ operations where $M$
is the number of collocation points.

The paper is organized as follows. In Section 2, we provide the high-order scheme with the collocation method for
solving nonlocal problems. In Section 3, the
superconvergence rate with the Crank-Nicolson scheme is studied.
In Section 4, we develop conjugate gradient squared method
to solving the discretized linear system and discuss the computational cost and the storage requirement.
In Section 5, experimental results are given to illustrate the effectiveness of the proposed numerical method.
Finally, some concluding remarks are given in Section 6.

\section{Discretization Schemes}
In this section, we discuss about the discretization schemes of the nonlocal problems including two-dimensional cases.
\subsection{One-dimensional Discretization}
Review the piecewise quadratic polynomial collocation method and apply it to the following
steady-state version of \eqref{1.1}:
\begin{equation}\label{2.1}
 \int^b_a \frac{u(x)-u(y)}{|x-y|^\gamma}dy  =f(x),   \quad 0< \gamma <1.
\end{equation}
Let the mesh points $a=x_{0}<x_{\frac{1}{2}}<x_{1}< \cdots <x_{\frac{2M-1}{2}}<x_{M}=b$ be
a partition with the uniform spatial stepsize $h=(b-a)/M$
and $0=t_{0}<t_1< \cdots <t_{N}=T$ with the time stepsize $\tau=T/N$.
Denote $u_{i}^k$ as the approximated value of $u(x_i,t_k)$ and $f_i^{k+1/2}=f(x_i,t_{k+1/2})$ with $t_{k+1/2}=\frac{t_k+t_{k+1}}{2}$.

Let the piecewise quadratic basis function $\phi_{i}(y)$ or $\phi_{i+\frac{1}{2}}(y)$ be given in \cite[p.\,499]{Aikinson:09}.
Then the piecewise Lagrange quadratic interpolant of $u(y)$ is
$$
u_{Q}(y)=\sum^{M}_{i=0}u(x_i)\phi_i(y)+\sum^{M-1}_{i=0}u\left(x_{i+\frac{1}{2}}\right)\phi_{i+\frac{1}{2}}(y).
$$
According to (2.9) from \cite{CQSW:19}, we can rewrite \eqref{2.1} as follows:
\begin{equation}\label{2.2}
\begin{split}
\int^{b}_{a} \frac{ u\left(x_\frac{i}{2}\right)-u_{Q}(y)}{\left| x_{\frac{i}{2}} - y \right|^{\gamma}} dy
= f\left(x_\frac{i}{2}\right)+R_\frac{i}{2}, \quad 1\leq i\leq 2M-1,
\end{split}
\end{equation}
with $R_\frac{i}{2}=\mathcal{O}\left(h^4\left(\eta_\frac{i}{2}\right)^{-\gamma}\right)$ and  $\eta_\frac{i}{2}=\min\left\{x_{\frac{i}{2}}-a,b-x_{\frac{i}{2}}\right\}$ for $i=1,2,\cdots,2M-1$.
Thus the discretization scheme of \eqref{2.2} is given by the following system:
\begin{equation*}
\begin{split}
&\left[d_{i} u_{i}-\sum^{M-1}_{j=1}m_{|i-j|}u_j-\sum^{M-1}_{j=0}q_{|i-j-\frac{1}{2}|-\frac{1}{2}}u_{j+\frac{1}{2}}\right]
=f_{i}+\left(\beta_iu_0+\beta_{M-i}u_M\right),\\
&{\rm with}\ i=1,2,\cdots,M-1,\, {\rm and}\\
&\left[d_{i+\frac{1}{2}} u_{i+\frac{1}{2}}-\sum^{M-1}_{j=1}p_{|i+\frac{1}{2}-j|-\frac{1}{2}}u_j-\sum^{M-1}_{j=0}n_{|i-j|}u_{j+\frac{1}{2}}\right]
=f_{i+\frac{1}{2}}+\left(\gamma_iu_0+\gamma_{M-i-1}u_M\right),\\
&{\rm with}\ i=0,1,2,\cdots,M-1.
\end{split}
\end{equation*}
Here, the coefficients  are given in (2.10) of \cite{CQSW:19}. For simplicity, we set
$\eta_{h,\gamma}$ to be $\frac{h^{1-\gamma}}{(3-\gamma)(2-\gamma)(1-\gamma)}$ and
explicitly compute $m_0=2(1+\gamma){\eta_{h,\gamma}}$, for $ k\geq 1$,
$$
m_k=4{\eta_{h,\gamma}}\left[ (k+1)^{3-\gamma}\!-\!(k-1)^{3-\gamma} \right]
\!-\!{\eta_{h,\gamma}}(3-\gamma)\left[ (k+1)^{2-\gamma}+6k^{2-\gamma}+(k-1)^{2-\gamma}  \right];
$$
and
$p_0=4{\eta_{h,\gamma}}\left[\left(\frac{3}{2}\right)^{3-\gamma}-\left(\frac{1}{2}\right)^{3-\gamma}\right]
-{\eta_{h,\gamma}}(3-\gamma)\left[\left(\frac{3}{2}\right)^{2-\gamma}+3\left(\frac{1}{2}\right)^{2-\gamma}\right]$,
for $k \geq 1$, $p_{k}=m_{k+\frac{1}{2}}$.
Moreover $n_{0}={\eta_{h,\gamma}}{ (2-\gamma) 2^{\gamma+1}}$, $n_{k}=q_{k-\frac{1}{2}}$ for $k\geq 1$, and there exists $$q_{k}=-8{\eta_{h,\gamma}}\left((k+1)^{3-\gamma}-k^{3-\gamma}\right)+4{\eta_{h,\gamma}}(3-\gamma)\left((k+1)^{2-\gamma}+k^{2-\gamma}\right),\, k \ge 0.$$
The boundary coefficients for $1\ \leq i \leq M-1$ are given by
$$\beta_i
= 4{\eta_{h,\gamma}}\left[ i^{3-\gamma} \!- \!(i-1)^{3-\gamma}  \right]
-{\eta_{h,\gamma}}(3-\gamma) \left[ 3i^{2-\gamma} + \left(i-1\right)^{2-\gamma} - (2-\gamma)i^{1-\gamma}\right],$$
$\gamma_{i}=\beta_{i+\frac{1}{2}}$ and $\gamma_{0}={\eta_{h,\gamma}}{{(2-\gamma)(1-\gamma)} 2^{\gamma-1}}$.

Using the matrix form of the grid functions
$$
U=\left(u_{1},u_{2},\cdots,u_{M-1},u_{\frac{1}{2}},u_{\frac{3}{2}},\cdots,u_{M-\frac{1}{2}}\right)^{T}
$$
and similarly for $F$. We can rewrite the above discretization scheme as follows:
\begin{equation}\label{2.3}
\begin{split}
\mathcal{A}U=F+K~~ {\rm with}~~
\mathcal{A}=\left ( \begin{matrix}
 \mathcal{D}_{1}  & 0\\
 0      &\mathcal{D}_{2}
 \end{matrix}
 \right )
-
\left ( \begin{matrix}
  \mathcal{M}     & \mathcal{Q }  \\
 \mathcal{ P }        & \mathcal{N}
 \end{matrix}
 \right )=: {\cal D} - {\cal G},
\end{split}
\end{equation}
where the boundary data $K$ is given by
\begin{equation*}
\begin{split}
  K&=\left(\eta_{1},\eta_{2},\cdots,\eta_{M-1},\eta_\frac{1}{2},\eta_\frac{3}{2},\cdots,\eta_{M-\frac{1}{2}}\right)^{T}u_0\\
&\quad+\left(\eta_{M-1},\eta_{M-2},\cdots,\eta_{1},\eta_{M-\frac{1}{2}},\eta_{M-\frac{3}{2}},\cdots,\eta_{\frac{1}{2}}\right)^{T}u_M.
\end{split}
\end{equation*}
Moreover, $\mathcal{D}_1={\rm diag}\left(d_1,d_2,\ldots, d_{M-1}\right)$, $\mathcal{D}_2={\rm diag}\left(d_\frac{1}{2},d_\frac{3}{2},\ldots, d_{M-\frac{1}{2}}\right)$,
 and $\mathcal{M}={\rm toeplitz}\left(m_0,m_1,\ldots, m_{M-2}\right)$, $\mathcal{N}={\rm toeplitz}\left(n_0,n_1,\ldots, n_{M-1}\right)$.
The {\em rectangular matrices} $\mathcal{P}$, $\mathcal{Q}$ are defined by
\begin{equation*}
\begin{split}
\mathcal{P}=\left [ \begin{matrix}
p_{0}              & p_{1}              & p_{2}             &     \cdots & p_{M-3}   & p_{M-2}        \\
p_{0}              & p_{0}              & p_{1}             &     \ddots & \ddots   & p_{M-3}         \\
p_{1}              & p_{0}              & p_{0}             &     \ddots & \ddots      & \vdots          \\
\vdots             & \ddots              &  \ddots           &     \ddots & \ddots      & p_{2}            \\
p_{M-4}            &\ddots           & \ddots         &     \ddots & p_{0}       & p_{1}              \\
p_{M-3}          & p_{M-4}           &\ddots         &      p_{1} & p_{0}       & p_{0}               \\
p_{M-2}          & p_{M-3}           & p_{M-4}        &     \cdots & p_{1}       & p_{0}
 \end{matrix}
 \right ]_{M \times (M-1)}
\end{split},
\end{equation*}
and
\begin{equation*}
\begin{split}
\mathcal{Q}=\left [ \begin{matrix}
q_{0}              & q_{0}              & q_{1}             &     \cdots & q_{M-4}   & q_{M-3}   & q_{M-2}      \\
q_{1}              & q_{0}              & q_{0}             &     \ddots & \ddots   & q_{M-4}   & q_{M-3}       \\
q_{2}              & q_{1}              & q_{0}             &     \ddots & \ddots   & \ddots   & q_{M-4}        \\
\vdots             & \ddots             &  \ddots           &     \ddots & \ddots    &q_{1}     & \vdots           \\
q_{M-3}            & \ddots            & \ddots           &     \ddots & q_{0}     & q_{0}     &  q_{1}            \\
q_{M-2}            & q_{M-3}            & \cdots         &     q_{2} & q_{1}     & q_{0}     &  q_{0}
 \end{matrix}
 \right ]_{(M-1) \times M}
\end{split}.
\end{equation*}

Hence, the full discretization of time-dependent non-local problems in (1.1) with Crank-Nicolson scheme is given by
\begin{equation}\label{2.4}
\left(I+\frac{\tau}{2}\mathcal {A}\right)U^{k}=\left(I-\frac{\tau}{2}\mathcal {A}\right)U^{k-1}+\tau F^{k-\frac{1}{2}}+\tau K^{k-\frac{1}{2}},\, k=1,2,\cdots,N,
\end{equation}
with $U^k=\left(u^k_{1},u^k_{2},\cdots,u^k_{M-1},u^k_{\frac{1}{2}},u^k_{\frac{3}{2}},\cdots,u^k_{M-\frac{1}{2}}\right)^{T}$.
Note that the local truncation error is of $\mathcal {O}\left(\tau^2+h^4\left(\eta_\frac{i}{2}\right)^{-\gamma}\right)$ with
$\eta_\frac{i}{2}=\min\left\{x_{\frac{i}{2}}-a,b-x_{\frac{i}{2}}\right\}$.

\subsection{Two-dimensional Nonlocal Problems with Multiplicative Cauchy Kernel}\label{subsection 2.2}
As one of the two-dimensional nonlocal problems, we consider the following nonlocal problem with multiplicative Cauchy kernel:
\begin{equation}\label{2.5}
 \int_\Omega \frac{u(x,y)-u(\bar{x},\bar{y})}{|x-\bar{x}|^\gamma|y-\bar{y}|^\gamma}d\bar{x}d\bar{y}  =f(x,y),  \quad 0 < \gamma <1,
\end{equation}
with $\Omega=\left(a,b\right)\times \left(c,d\right)$. Taking the mesh points $a=x_{0}<x_{\frac{1}{2}}<x_{1}< \cdots <x_{\frac{2M_x-1}{2}}<x_{M_x}=b$ and $c=y_{0}<y_{\frac{1}{2}}<y_{1}< \cdots <y_{\frac{2M_y-1}{2}}<y_{M_y}=d$ as a partition with the uniform  space stepsize $h_x=(b-a)/M_x$ in $x$ direction, and $h_y=(d-c)/M_y$ in $y$ direction, and $t_{k}=k\tau$ with the time stepsize $\tau=T/N$. Denote  $u_{i,j}^k$ as the approximated value of $u\left(x_i,y_j,t_k\right)$ and $f_{i,j}^{k+1/2}=f\left(x_i,y_j,t_{k+1/2}\right)$ with $t_{k+1/2}=\frac{t_k+t_{k+1}}{2}$.
From \cite[p.\,499]{Aikinson:09}, we known that the piecewise quadratic basis function $\phi_{l}(x)$ or $\phi_{l-\frac{1}{2}}(x)$ are defined by
\begin{equation}\label{2.6}
\phi_l(x)=\left\{\begin{split}
\frac{x-x_{l-1}}{h_x}\frac{2x-(x_l+x_{l-1})}{h_x}:=\phi_l^{-}(x),\quad x&\in\left[x_{l-1},x_l\right],\\
\frac{x_{l+1}-x}{h_x}\frac{(x_{l+1}+x_l)-2x}{h_x}:=\phi_l^{+}(x),\quad x&\in\left[x_l,x_{l+1}\right],\\
0,\quad\qquad\qquad\qquad\qquad\qquad &{\rm otherwise}
\end{split}\right.
\end{equation}
with $l=1,2,\cdots,M_x-1$ and
\begin{equation}\label{2.7}
\phi_{l-\frac{1}{2}}(x)=\left\{\begin{split}
\frac{4(x-x_{l-1})(x_l-x)}{{h_x}^2},\qquad\ &x\in\left[x_{l-1},x_l\right],\\
0,\quad\qquad\qquad\qquad\qquad & {\rm otherwise}
\end{split}\right.
\end{equation}
with $l=1,2,\cdots,M_x$. The piecewise Lagrange quadratic interpolation of  $u\left(\bar{x},\bar{y}\right)$ is
\begin{equation}\label{2.8}
u_{Q}(\bar{x},\bar{y})=\sum^{2M_x}_{l=0}\sum^{2M_y}_{r=0}\phi_{\frac{l}{2},\frac{r}{2}}(\bar{x},\bar{y})u\left(x_{\frac{l}{2}},y_{\frac{r}{2}}\right)
=\sum^{2M_x}_{l=0}\phi_{\frac{l}{2}}(\bar{x})\sum^{2M_y}_{r=0}\phi_{\frac{r}{2}}(\bar{y})u\left(x_{\frac{l}{2}},y_{\frac{r}{2}}\right).
\end{equation}
Hence, for  $1\leq i\leq 2M_x-1$, $1\leq j\leq 2M_y-1$, we can rewrite \eqref{2.5} as
 \begin{equation}\label{2.9}
 \int_\Omega \frac{u\left(x_\frac{i}{2},y_\frac{j}{2}\right)-u_{Q}\left(\bar{x},\bar{y}\right)}{|x_\frac{i}{2}-\bar{x}|^\gamma|y_\frac{j}{2}-\bar{y}|^\gamma}d\bar{x}d\bar{y}  =f\left(x_\frac{i}{2},y_\frac{j}{2}\right)+R_{\frac{i}{2},\frac{j}{2}},
\end{equation}
where  the error estimation $R_{\frac{i}{2},\frac{j}{2}}$ will be proved in Lemma \ref{lemma3.9}. Thus the discretization scheme of \eqref{2.9} can be expressed as
\begin{equation*}
\begin{split}
&\int^b_a\frac{1}{|x_\frac{i}{2}-\bar{x}|^\gamma}d\bar{x}
\int^d_c\frac{1}{|y_\frac{j}{2}-\bar{y}|^\gamma}d\bar{y}\ u_{\frac{i}{2},\frac{j}{2}}\\
&\quad-\sum^{2M_x-1}_{l=1}\int^b_a \frac{\phi_{\frac{l}{2}}(\bar{x})}{|x_\frac{i}{2}-\bar{x}|^\gamma}d\bar{x}
\sum^{2M_y-1}_{r=1}\int^d_c\frac{\phi_{\frac{r}{2}}(\bar{y})}{|y_\frac{j}{2}-\bar{y}|^\gamma}d\bar{y}\ u_{\frac{l}{2},\frac{r}{2}} =f_{\frac{i}{2},\frac{j}{2}}+k_{\frac{i}{2},\frac{j}{2}},
\end{split}
\end{equation*}
and   the boundary data $k_{\frac{i}{2},\frac{j}{2}}$ is
\begin{equation*}
\begin{split}
&\int^b_a \!\frac{\phi_{0}(\bar{x})}{|x_\frac{i}{2}-\bar{x}|^\gamma}d\bar{x}
\sum^{2M_y}_{r=0}\!\int^d_c\!\frac{\phi_{\frac{r}{2}}(\bar{y})}{|y_\frac{j}{2}-\bar{y}|^\gamma}d\bar{y}\ u_{0,\frac{r}{2}}
\!+\!\int^b_a \!\frac{\phi_{M_x}(\bar{x})}{|x_\frac{i}{2}-\bar{x}|^\gamma}d\bar{x}
\sum^{2M_y}_{r=0}\!\int^d_c\!\frac{\phi_{\frac{r}{2}}(\bar{y})}{|y_\frac{j}{2}-\bar{y}|^\gamma}d\bar{y}\ u_{M_x,\frac{r}{2}}\\
&\!+\!\!\!\sum^{2M_x-1}_{l=1}\!\!\int^b_a\!\! \frac{\phi_{\frac{l}{2}}(\bar{x})}{|x_\frac{i}{2}-\bar{x}|^\gamma}d\bar{x}
\!\!\int^d_c\!\!\frac{\phi_{0}(\bar{y})}{|y_\frac{j}{2}-\bar{y}|^\gamma}d\bar{y}\ u_{\frac{l}{2},0}
\!+\!\!\!\sum^{2M_x-1}_{l=1}\!\!\int^b_a\! \!\frac{\phi_{\frac{l}{2}}(\bar{x})}{|x_\frac{i}{2}-\bar{x}|^\gamma}d\bar{x}
\!\!\int^d_c\!\!\frac{\phi_{M_y}(\bar{y})}{|y_\frac{j}{2}-\bar{y}|^\gamma}d\bar{y}\ u_{\frac{l}{2},M_y}.
\end{split}
\end{equation*}

For convenience of implementation, we define  following grid functions
\begin{equation}\label{2.10}
\begin{split}
{U}&=\left( U_1,U_2,\cdots,U_{M_x-1},U_{\frac{1}{2}},U_{\frac{3}{2}},\cdots,U_{M_x\!-\frac{1}{2}} \right)^T,\\
U_i&=\left(u_{i,1},u_{i,2},\ldots,u_{i,M_y-1},u_{i,\frac{1}{2}},u_{i,\frac{3}{2}},\ldots,u_{i, M_y-\frac{1}{2}}\right),
\end{split}
\end{equation}
with $i=1,2,\ldots, M_x-1,\textstyle\frac{1}{2},\frac{3}{2},\cdots,M_x-\frac{1}{2}$,
and similarly denote ${F}$, ${K}$.
Then  we can obtain the resulting system  of \eqref{2.9}
\begin{equation}\label{2.11}
 {\mathcal{A}}{U}={F}+{K}~~{\rm with}~~
 {\mathcal{A}}=\mathcal{D}_x\otimes \mathcal{D}_y - \mathcal{G}_x\otimes \mathcal{G}_y,
\end{equation}
in which  $\mathcal{D}_x$, $\mathcal{D}_y$,  $\mathcal{G}_x$, $\mathcal{G}_y$  of the same form as the matrix $\mathcal{D}$ and $\mathcal{G}$ in  \eqref{2.3}.

Consider the following  two-dimensional  time-dependent nonlocal problem
\begin{equation}\label{2.12}
\frac{\partial u(x,y,t)}{\partial t}+ \int_\Omega \frac{u(x,y,t)-u(\bar{x},\bar{y},t)}{|x-\bar{x}|^\gamma|y-\bar{y}|^\gamma}d\bar{x}d\bar{y}=f(x,y,t),\ (x,y,t)\in\Omega\times\left(0,T\right],
\end{equation}
with the nonhomogeneous   Dirichlet  boundary conditions  and the initial condition.

Using the grid functions
\begin{equation*}
\begin{split}
{U}^k
&=\left( U^k_1,U^k_2,\cdots,U^k_{M_x-1},U^k_{\frac{1}{2}},U^k_{\frac{3}{2}},\cdots,U^k_{M_x\!-\frac{1}{2}} \right)^T,\\
U^k_i
&=\left(u^k_{i,1},u^k_{i,2},\ldots,u^k_{i,M_y-1},u^k_{i,\frac{1}{2}},u^k_{i,\frac{3}{2}},\ldots,u^k_{i, M_y-\frac{1}{2}}\right),
\end{split}
\end{equation*}
with $ i=1,2,\ldots, M_x-1,\textstyle\frac{1}{2},\frac{3}{2},\cdots,M_x-\frac{1}{2}$ for $k=1,2,\cdots,N$.  Similarly, we can define the vectors  ${F}^k$ and ${K}^k$.
Then the full discretization of time-dependent nonlocal problems \eqref{2.12} is
\begin{equation}\label{2.13}
\left(I+\frac{\tau}{2}{\mathcal{A}}\right){U}^{k}=\left(I-\frac{\tau}{2}{\mathcal{A}}\right){U}^{k-1}+\tau {F}^{k-\frac{1}{2}}+\tau {K}^{k-\frac{1}{2}}.
\end{equation}

\subsection{Numerical scheme for 2D  nonlocal problems with additive Cauchy kernels}\label{subsection 2.3}
Consider the following  two-dimensional steady-state nonlocal problem
\begin{equation}\label{2.14}
 \int_\Omega \frac{u(x,y)-u(\bar{x},\bar{y})}{\left|\sqrt{\left(x-\bar{x}\right)^2+\left(y-\bar{y}\right)^2}\right|^{\gamma}}d\bar{x}d\bar{y}  =f(x,y),   \quad 0< \gamma <1.
\end{equation}
From \eqref{2.8}, for $1\leq i\leq 2M_x-1$,  $1\leq j \leq 2M_y-1$,  we can rewrite \eqref{2.14} as
 \begin{equation}\label{2.15}
 \int_\Omega \frac{u\left(x_\frac{i}{2},y_\frac{j}{2}\right)-u_{Q}\left(\bar{x},\bar{y}\right)}{\left|\sqrt{\left(x_\frac{i}{2}-\bar{x}\right)^2+\left(y_\frac{j}{2}-\bar{y}\right)^2}\right|^{\gamma}}
 d\bar{x}d\bar{y}  =f\left(x_\frac{i}{2},y_\frac{j}{2}\right)+R_{\frac{i}{2},\frac{j}{2}},
\end{equation}
where  the error estimation $R_{\frac{i}{2},\frac{j}{2}}$ will be proved in Lemma \ref{lemma3.12}.
Then the discretization scheme of \eqref{2.15} can be expressed by
\begin{equation}\label{2.16}
\begin{split}
&\int_\Omega\frac{1}{\left|\sqrt{\left(x_\frac{i}{2}-\bar{x}\right)^2+\left(y_\frac{j}{2}-\bar{y}\right)^2}\right|^{\gamma}} d\bar{x}d\bar{y} \ u_{\frac{i}{2},\frac{j}{2}}\\
&\quad-\sum^{2M_x-1}_{l=1}\sum^{2M_y-1}_{r=1}\int_\Omega\frac{\phi_{\frac{l}{2},\frac{r}{2}}(\bar{x},\bar{y}) }{\left|\sqrt{\left(x_\frac{i}{2}-\bar{x}\right)^2+\left(y_\frac{j}{2}-\bar{y}\right)^2}\right|^{\gamma}}d\bar{x}d\bar{y} \ u_{\frac{l}{2},\frac{r}{2}} =f_{\frac{i}{2},\frac{j}{2}}+k_{\frac{i}{2},\frac{j}{2}},
\end{split}
\end{equation}
 and boundary data $k_{\frac{i}{2},\frac{j}{2}}$ is
\small\begin{equation*}
\begin{split}
&\sum^{2M_y}_{r=0}\int_\Omega\frac{\phi_{0,\frac{r}{2}}(\bar{x},\bar{y})~ d\bar{x}d\bar{y}}
{\left|\sqrt{\left(x_\frac{i}{2}-\bar{x}\right)^2+\left(y_\frac{j}{2}-\bar{y}\right)^2}\right|^{\gamma}}\ u_{0,\frac{r}{2}}
+\sum^{2M_y}_{r=0}\int_\Omega\frac{\phi_{M_x,\frac{r}{2}}(\bar{x},\bar{y})~ d\bar{x}d\bar{y}}
{\left|\sqrt{\left(x_\frac{i}{2}-\bar{x}\right)^2+\left(y_\frac{j}{2}-\bar{y}\right)^2}\right|^{\gamma}}\ u_{M_x,\frac{r}{2}}\\
&\!\!+\!\!\sum^{2M_x-1}_{l=1}\!\!\!\int_\Omega\frac{\phi_{\frac{l}{2},0}(\bar{x},\bar{y})~d\bar{x}d\bar{y}}
{\left|\sqrt{\left(x_\frac{i}{2}-\bar{x}\right)^2+\left(y_\frac{j}{2}-\bar{y}\right)^2}\right|^{\gamma}}\ u_{\frac{l}{2},0}
+\!\!\sum^{2M_x-1}_{l=1}\!\!\!\int_\Omega\frac{\phi_{\frac{l}{2},M_y}(\bar{x},\bar{y})~d\bar{x}d\bar{y}}
{\left|\sqrt{\left(x_\frac{i}{2}-x\right)^2+\left(y_\frac{j}{2}-\bar{y}\right)^2}\right|^{\gamma}}\ u_{\frac{l}{2},M_y}.
\end{split}
\end{equation*}\small

For convenience of implementation, applying the  grid functions in \eqref{2.10} again,   we  obtain the algebraic equation of \eqref{2.14}
\begin{equation}\label{2.17}
 {\mathcal{A}}{U}={F}+{K}~~{\rm with}~~
  {\mathcal{A}}={\mathcal{D}}-{\mathcal{G}} ~~{\rm and}~~{\mathcal{G}}=\left(\begin{array}{cc}{\mathcal{M}} & {\mathcal{Q}}\\ {\mathcal{P}} & {\mathcal{N}}\end{array}\right).
\end{equation}
Here $\mathcal{M}$, ${\mathcal{Q}}$, ${\mathcal{P}}$ and ${\mathcal{N}}$ are different from which in \eqref{2.3}. Denote the kernel function $\delta(x_i-\bar{x},y_j-\bar{y}):=\left|\sqrt{\left(x_i-\bar{x}\right)^2+\left(y_j-\bar{y}\right)^2}\right|^{\gamma}$, the entries of ${\mathcal{D}}$ are $${\mathcal{D}}({i,j})=\int_{c}^{d}\int_{a}^{b}\frac{1}{\delta(x_i-\bar{x},y_j-\bar{y})}d\bar{x}d\bar{y}$$
with $i=1,2,\cdots, M_x-1,\textstyle\frac{1}{2},\frac{3}{2},\cdots,M_x-\frac{1}{2}$, $j=1,2,\cdots, M_y-1,\textstyle\frac{1}{2},\frac{3}{2},\cdots,M_y-\frac{1}{2}$. It should be noted that the coefficients of the variables in \eqref{2.16} can not  be computed directly, but it can be verified to have the block-Toeplitz properties by the coefficients expression in (i), (ii), (iii) and (iv) later. And the matrix  ${\mathcal{G}}$  consists of four block-structured matrices with Toeplitz-like blocks, and the block-Toeplitz properties of the ${\mathcal{M}}_{\left(M_x-1\right)\times\left(M_x-1\right)}$  can be expressed as follows:
\begin{equation*}
{\mathcal{M}}=\left(\begin{array}{cccccc}
{\mathcal{M}}_{1,1}        & {\mathcal{M}}_{1,2} & \cdots                          & {\mathcal{M}}_{1,M_x-2}  &{\mathcal{M}}_{1,M_x-1}\\[2mm]
{\mathcal{M}}_{2,1}        & {\mathcal{M}}_{1,1} & {\mathcal{M}}_{1,2} & \ddots                                  &{\mathcal{M}}_{1,M_x-2}\\[2mm]
\vdots                                 & {\mathcal{M}}_{2,1} & {\mathcal{M}}_{1,1} & \ddots                                  &\vdots\\[2mm]
{\mathcal{M}}_{M_x-2,1} & \ddots                           &\ddots                            &\ddots                                   &{\mathcal{M}}_{1,2}\\[2mm]
{\mathcal{M}}_{M_x-1,1} &{\mathcal{M}}_{M_x-2,1} &\cdots                       &{\mathcal{M}}_{2,1}         &{\mathcal{M}}_{1,1}
\end{array}\right),
\end{equation*}
in the same way, ${\mathcal{Q}}_{\left(M_x-1\right)\times M_x }$, ${\mathcal{P}}_{ M_x \times\left(M_x-1\right)}$ and ${\mathcal{N}}_{ M_x \times  M_x }$ are expressed in Appendix A.

Each block of ${\mathcal{M}}$ numbered with $i,l=1,2,\cdots, M_x-1$ has the following form
\begin{equation}\label{2.18}
{\mathcal{M}}_{i,l}=\left(\begin{array}{cccc}
{\mathcal{M}}^{\mathcal{M}}_{i,l} & {\mathcal{M}}^{\mathcal{Q}}_{i,l} \\[2mm]
{\mathcal{M}}^{\mathcal{P}}_{i,l} & {\mathcal{M}}^{\mathcal{N}}_{i,l}
\end{array}\right)_{\left(2M_y-1\right)\times\left(2M_y-1\right)}~~~~{\rm with}~~ {\mathcal{M}}_{i,l}={\mathcal{M}}_{l,i},
\end{equation}
and the entries of ${\mathcal{M}}^{\mathcal{M}}_{i,l}$, ${\mathcal{M}}^{\mathcal{P}}_{i,l}$, ${\mathcal{M}}^{\mathcal{Q}}_{i,l}$ and ${\mathcal{M}}^{\mathcal{N}}_{i,l}$ numbered with $j$ and $r$ can be expressed as
\begin{itemize}
  \item
  ${\mathcal{M}}^{\mathcal{M}}_{i,l}(j,r)=C_{i,j}^{l,r}$,\qquad\quad\ for $ j=1,2,\cdots, M_y-1,\, r=1,2,\cdots, M_y-1$;
  \item
  ${\mathcal{M}}^{\mathcal{P}}_{i,l}(j,r)=C_{i,j}^{l,r}$,\qquad\quad\ \,for $j=\textstyle\frac{1}{2},\frac{3}{2},\cdots,M_y-\frac{1}{2},\, r=1,2,\cdots, M_y-1$;
  \item
  ${\mathcal{M}}^{\mathcal{Q}}_{i,l}(j,r-\frac{1}{2})=C_{i,j}^{l,r-\frac{1}{2}}$,\ for $j=1,2,\cdots, M_y-1,\, r=1,2,\cdots, M_y$;
  \item
  ${\mathcal{M}}^{\mathcal{N}}_{i,l}(j,r-\frac{1}{2})=C_{i,j}^{l,r-\frac{1}{2}}$,\ for $j=\textstyle\frac{1}{2},\frac{3}{2},\cdots,M_y-\frac{1}{2},\, r=1,2,\cdots, M_y$.
\end{itemize}
Here the above coefficients  can be computed  by the following 

(i) for $ l=1,2,\cdots, M_x-1,\ r=1,2,\cdots, M_y-1$,
\begin{equation*}
\begin{split}
C_{i,j}^{l,r}=&\int_{y_{r-1}}^{y_r}\int_{x_{l-1}}^{x_l}\frac{\phi_l^{-}(\bar{x})\phi_r^{-}(\bar{y})}{\delta(x_i-\bar{x},y_j-\bar{y})}d\bar{x}d\bar{y}
+\int_{y_{r}}^{y_{r+1}}\int_{x_{l-1}}^{x_l}\frac{\phi_l^{-}(\bar{x})\phi_r^{+}(\bar{y})}{\delta(x_i-\bar{x},y_j-\bar{y})}d\bar{x}d\bar{y}\\
&+\int_{y_{r-1}}^{y_r}\int_{x_{l}}^{x_{l+1}}\frac{\phi_l^{+}(\bar{x})\phi_r^{-}(\bar{y})}{\delta(x_i-\bar{x},y_j-\bar{y})}d\bar{x}d\bar{y}
+\int_{y_{r}}^{y_{r+1}}\int_{x_{l}}^{x_{l+1}}\frac{\phi_l^{+}(\bar{x})\phi_r^{+}(\bar{y})}{\delta(x_i-\bar{x},y_j-\bar{y})}d\bar{x}d\bar{y};
\end{split}
\end{equation*}

(ii) for $l=1,2,\cdots, M_x-1,\ r=1,2,\cdots, M_y$,
\begin{equation*}
C_{i,j}^{l,r-\frac{1}{2}}=\int_{y_{r-1}}^{y_r}\int_{x_{l-1}}^{x_l}\frac{\phi_l^{-}(\bar{x})\phi_{r-\frac{1}{2}}(\bar{y})}{\delta(x_i-\bar{x},y_j-\bar{y})}d\bar{x}d\bar{y}
+\int_{y_{r-1}}^{y_{r}}\int_{x_{l}}^{x_{l+1}}\frac{\phi_l^{-}(\bar{x})\phi_{r-\frac{1}{2}}(\bar{y})}{\delta(x_i-\bar{x},y_j-\bar{y})}d\bar{x}d\bar{y};
\end{equation*}

(iii) for $ l=1,2,\cdots, M_x,\ r=1,2,\cdots, M_y-1$,
\begin{equation*}
C_{i,j}^{l-\frac{1}{2},r}=\int_{y_{r-1}}^{y_r}\int_{x_{l-1}}^{x_{l}}\frac{\phi_{l-\frac{1}{2}}(\bar{x})\phi_r^{-}(\bar{y})}{\delta(x_i-\bar{x},y_j-\bar{y})}d\bar{x}d\bar{y}
+\int_{y_{r}}^{y_{r+1}}\int_{x_{l-1}}^{x_{l}}\frac{\phi_{l-\frac{1}{2}}(\bar{x})\phi_r^{+}(\bar{y})}{\delta(x_i-\bar{x},y_j-\bar{y})}d\bar{x}d\bar{y};
\end{equation*}

(iv) for $l=1,2,\cdots, M_x,\ r=1,2,\cdots, M_y$,
\begin{equation*}
C_{i,j}^{l-\frac{1}{2},r-\frac{1}{2}}=\int_{y_{r-1}}^{y_{r}}\int_{x_{l-1}}^{x_{l}}\frac{\phi_{l-\frac{1}{2}}(\bar{x})\phi_{r-\frac{1}{2}}(\bar{y})}{\delta(x_i-\bar{x},y_j-\bar{y})}d\bar{x}d\bar{y}.
\end{equation*}
Moreover, we have
\begin{flalign*}
\begin{split}
\quad &{\mathcal{M}}^{\mathcal{M}}_{i,l}(j,r)={\mathcal{M}}^{\mathcal{M}}_{i,l}(j+1,r+1),\qquad\ {\rm for}\ j=1,2,\cdots, M_y-1,\, r=1,2,\cdots, M_y-1,\\[2mm]
\quad  &{\mathcal{M}}^{\mathcal{P}}_{i,l}(j,r)={\mathcal{M}}^{\mathcal{P}}_{i,l}(j+1,r+1),\qquad\ \,{\rm for}\ j=\textstyle\frac{1}{2},\frac{3}{2},\cdots,M_y-\frac{1}{2},\, r=1,2,\cdots, M_y-1,\\[2mm]
\quad  &{\mathcal{M}}^{\mathcal{Q}}_{i,l}(j,r-\frac{1}{2})={\mathcal{M}}^{\mathcal{Q}}_{i,l}(j+1,r+\frac{1}{2}),\ {\rm for}\ j=1,2,\cdots, M_y-1,\, r=1,2,\cdots, M_y,\\[2mm]
\quad  &{\mathcal{M}}^{\mathcal{N}}_{i,l}(j,r-\frac{1}{2})={\mathcal{M}}^{\mathcal{N}}_{i,l}(j+1,r+\frac{1}{2}),\ {\rm for}\ j=\textstyle\frac{1}{2},\frac{3}{2},\cdots,M_y-\frac{1}{2},\, r=1,2,\cdots, M_y,
\end{split}&&
\end{flalign*}
where ${\mathcal{M}}_{i,l}^{\mathcal{M}}$, ${\mathcal{M}}_{i,l}^{\mathcal{N}}$ are symmetric Toeplitz matrices, 
and ${\mathcal{M}}_{i,l}^{\mathcal{Q}}$, ${\mathcal{M}}_{i,l}^{\mathcal{P}}$ are  rectangular  matrices.

Consider the following  two-dimensional  time-dependent nonlocal problem
\begin{equation}\label{2.19}
\frac{\partial u(x,y,t)}{\partial t}+ \int_\Omega \frac{u(x,y,t)-u(\bar{x},\bar{y},t)}{\left|\sqrt{\left(x-\bar{x}\right)^2+\left(y-\bar{y}\right)^2}\right|^{\gamma}}d\bar{x}d\bar{y}=f(x,y,t),
\ (x,y,t)\in\Omega\times\left(0,T\right],
\end{equation}
with the nonhomogeneous   Dirichlet  boundary conditions  and the  initial condition.

By the similar discussion in \eqref{2.13}, we have the following  Crank-Nicolson scheme
\begin{equation}\label{2.20}
\left(I+\frac{\tau}{2}{\mathcal{A}}\right){U}^{k}=\left(I-\frac{\tau}{2}{\mathcal{A}}\right){U}^{k-1}+\tau {F}^{k-\frac{1}{2}}+\tau {K}^{k-\frac{1}{2}}.
\end{equation}

We can similarly  construct  the following  $4$-step backward differentiation formula (BDF4)  \cite{Chen:13} to confirm the superconvergence results
with $O(M \log M)$ operations, i.e.,
 \begin{equation}\label{2.21}
{\rm BDF4}~~\left(\frac{25}{12}I\!+\!\tau\mathcal{A}\right)U^k
=4U^{k-1}\!-\!3U^{k-2}+\frac{4}{3}U^{k-3}-\frac{1}{4}U^{k-4}+\tau F^k+\tau K^k.
 \end{equation}

\section{Stability and convergence analysis}
In first  subsection, we  study the  spectral properties  for nonsymmetric and indefinite matrix with one-dimensional cases, the unconditionally stability and convergence analysis with Crank-Nicolson scheme are proved in the remainder subsections.
\subsection{Spectral analysis for nonsymmetric and indefinite matrix  $\mathcal{A}$ in 1D}
\begin{lemma}\cite[p.\,28]{Quarteroni:07}\label{lemma3.1}
A real matrix $A$ of order $n$ is positive definite  if and only if  its symmetric part $H=\frac{A+A^T}{2}$ is positive definite.
Let $H \in \mathbb{R}^{n\times n}$ be symmetric. Then $H$ is positive definite if and only if the eigenvalues of $H$ are positive.
\end{lemma}
\begin{lemma}\cite{{Varah:75}}\label{lemma3.2}
Assume $A$ is diagonally dominant by rows.
 Then $||A^{-1}||_{\infty}<\frac{1}{\alpha}$
with $\alpha=\min_{i}\left( |a_{i,i}|-\sum_{j\neq i}|a_{i,j}| \right).$
\end{lemma}
\begin{lemma}\cite{CQSW:19}\label{lemma3.3}
Let the matrices $\mathcal{A}$, $\mathcal{M}$,  $\mathcal{N}$, $\mathcal{P}$, $\mathcal{Q}$ be defined by \eqref{2.3}.
Then $ \mathcal{M}$, $ \mathcal{N}$,  $\mathcal{P}$,  $\mathcal{Q}$ are positive matrices.
Moreover, the matrix $\mathcal{A}$ is  strictly  diagonally dominant by rows.
\end{lemma}
\begin{lemma}\label{lemma3.4}
Let the matrix $\mathcal{A}$ be defined by \eqref{2.3}.
Then the diagonal entries of $\mathcal{A}$ are bounded.
\end{lemma}
\begin{proof}
Let ${\rm diag}\left(\mathcal{A}\right)=\left(a_1,a_2,\ldots, a_{M-1}, a_\frac{1}{2},a_\frac{3}{2},\ldots, a_{M-\frac{1}{2}}\right)$.  From Lemma \ref{lemma3.3} and \eqref{2.3}, we have
\begin{equation*}
0<a_\frac{i}{2}<\eta_{h,\gamma}d_\frac{i}{2} = \int^b_a\frac{1}{\left|x_\frac{i}{2}-y\right|^\gamma}dy
=\frac{1}{1-\gamma}\left[ \left(x_{\frac{i}{2}}-a\right)^{1-\gamma}+\left(b-x_{\frac{i}{2}}\right)^{1-\gamma} \right]
\leq \frac{2\left(b-a\right)^{1-\gamma}}{1-\gamma}.
\end{equation*}
The proof is completed.
\end{proof}

Let the condition number $\kappa_p\left(A\right)=\left|\left|A\right|\right|_p\left|\left|A^{-1}\right|\right|_p$ with  $p=1,2,\ldots \infty$. Then we have
\begin{lemma}\label{lemma3.5}
Let the matrix $\mathcal{A}$ be defined by \eqref{2.3}.
Then the condition number
$$\kappa_\infty\left(\mathcal{A}\right)=\left|\left|\mathcal{A}\right|\right|_\infty\left|\left|\mathcal{A}^{-1}\right|\right|_\infty= \mathcal {O}\left(N\right).$$
\end{lemma}
\begin{proof}
Form (4.1) of \cite{CQSW:19}, we have
\begin{equation*}
\begin{split}
&\int^{b}_{a} \frac{ 1}{\left| x_{i} - y \right|^{\gamma}} dy -
 \int^{b}_{a} \frac{ \sum^{N-1}_{j=1}  \phi_{j}(y)}{\left| x_{i} - y \right|^{\gamma}} dy
=  \int^{b}_{a} \frac{\phi_{0}(x) }{\left| x_{i} - y \right|^{\gamma}} dy +
    \int^{b}_{a} \frac{\phi_{N}(x) }{\left| x_i - y \right|^{\gamma}} dy=\sigma_{h,\gamma}\rho_i\\
&\quad \geq \frac{(2-\gamma)(1-\gamma)}{2}\sigma_{h,\gamma}\left[\frac{1}{i^{\gamma}}  +\frac{1}{\left(N-i\right)^{\gamma}}\right]
=\frac{h^{1-\gamma}}{2}\left[\frac{1}{i^{\gamma}}  +\frac{1}{\left(N-i\right)^{\gamma}}\right]\geq \frac{h^{1-\gamma}}{N^\gamma}=(b-a)^{-\gamma} h.
\end{split}
\end{equation*}
From the above inequality  and Lemma \ref{lemma3.2}, it yields $\left|\left|\mathcal{A}^{-1}\right|\right|_{\infty}<(b-a)^\gamma h^{-1}$.
Combine with Lemmas \ref{lemma3.3} and \ref{lemma3.4}, we obtain
$$\kappa_\infty\left(\mathcal{A}\right)=||\mathcal{A}||_\infty\left|\left|\mathcal{A}^{-1}\right|\right|_\infty\leq \frac{4\left(b-a\right)}{1-\gamma} h^{-1} =\mathcal {O}\left(N\right).$$
The proof is completed.
\end{proof}

\begin{remark}
From   Lemma \ref{lemma3.3} and Theorem $1.21$ of \cite[p.\,23]{Varga:00}, it yields  $\Re\left(\lambda(\mathcal{A})\right)>0$ and  $\mathcal{A}$  nonsingular.
However, from Lemma \ref{lemma3.1} and counter-example in Figure 1, it shows that
$$\min\left( \lambda\left(H\right) \right)<0,~~\max\left( \lambda\left(H\right) \right)>0~~{\rm with}~~H=\frac{\mathcal{A}+\mathcal{A}^T}{2},$$
 i.e.,
matrix $\mathcal{A}$ is a nonsymmetric and indefinite.
\end{remark}
\begin{figure}[h] \centering
  \begin{tabular}{cc}
      \includegraphics[width=0.45\textwidth]{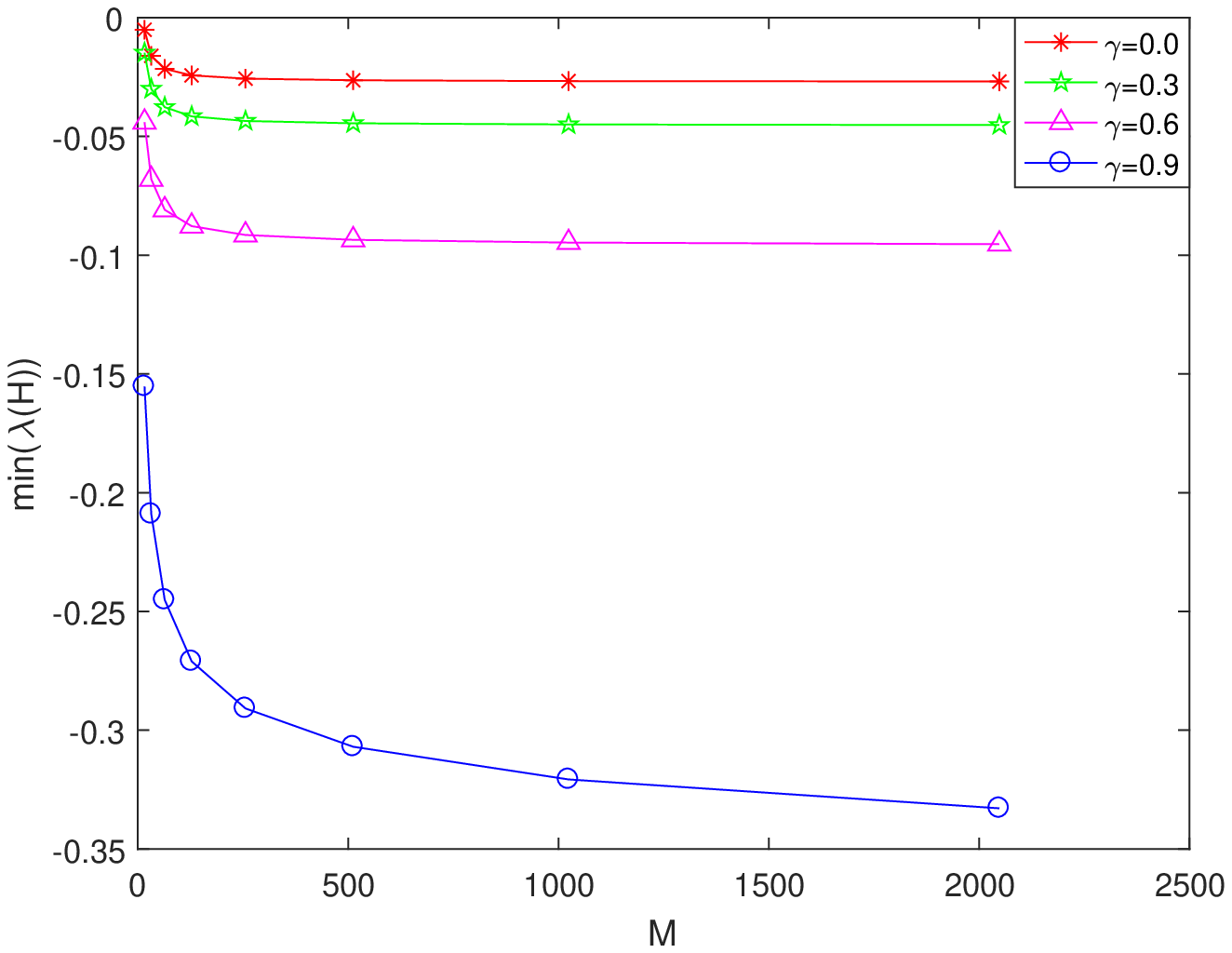}&\includegraphics[width=0.45\textwidth]{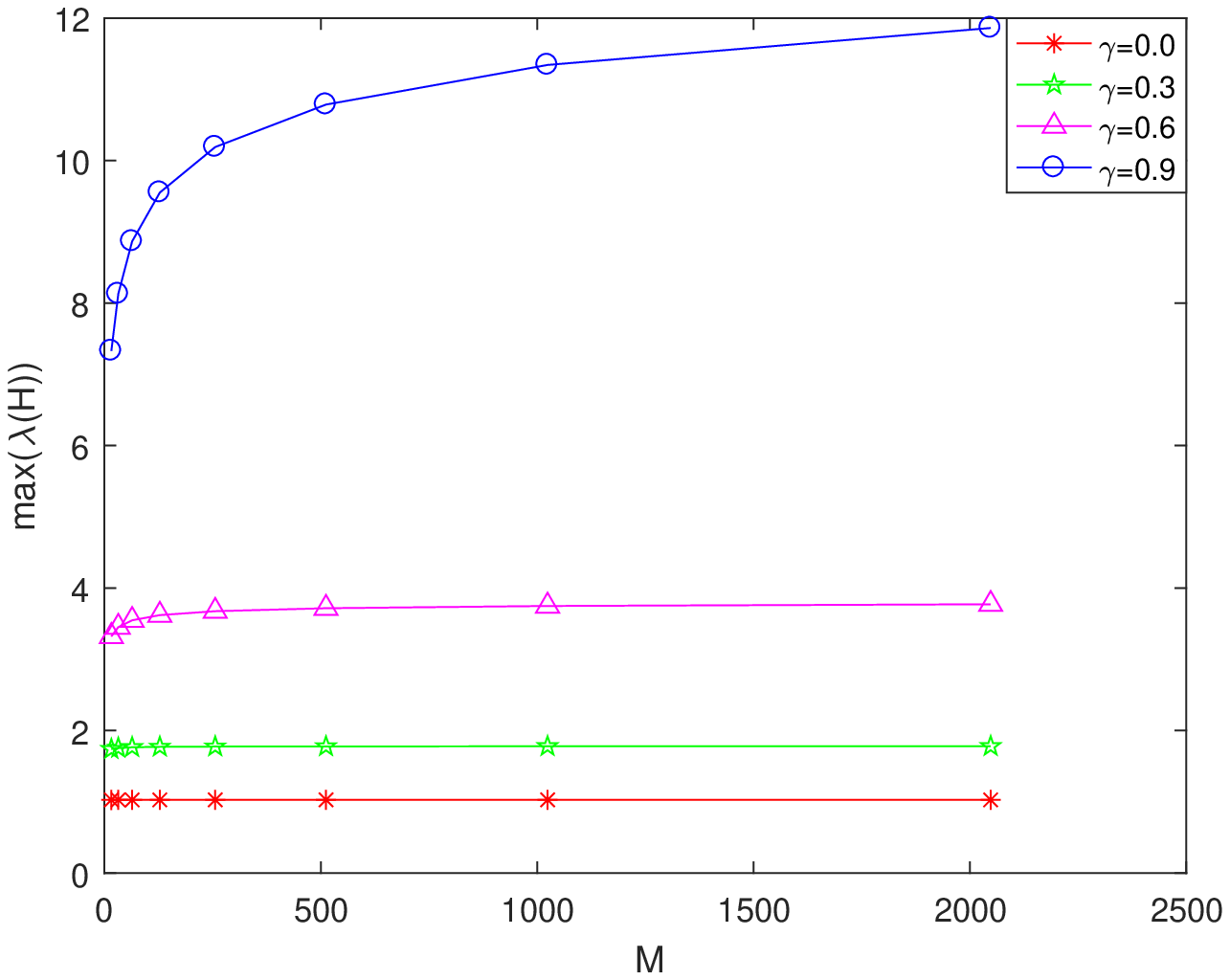}
  \end{tabular}
  \caption{The minimum and maximum  eigenvalues of $H=\frac{\mathcal{A}+\mathcal{A}^T}{2}$.}
  \label{figure01}
\end{figure}
\subsection{Stability and convergence analysis for Crank-Nicolson scheme \eqref{2.4} in 1D}
For steady-state nonlocal problem of \eqref{1.1}, an  optimal global convergence estimate with  $\mathcal{O}\left(h^3\right)$  was established in  \cite{CQSW:19}.
However,  for  the   time-dependent    problems of \eqref{1.1}, we next prove  the stability and convergence with the superconvergence  results.
\begin{theorem}\label{theorem3.6}
The numerical schemes \eqref{2.4} is unconditionally stable.
\end{theorem}
\begin{proof}
Let $\widetilde{u}_{i/2}^k~(i=1,2,\ldots,2M-1;\,k=0,1,\ldots,N)$ be the approximate solution of $u_{i/2}^k$,
which is the exact solution of the difference scheme (\ref{2.4}).
Putting $\epsilon_{i/2}^k=\widetilde{u}_{i/2}^k-u_{i/2}^k$, then using (\ref{2.4}), we obtain the following perturbation equation
\begin{equation*}
\left(I+\frac{\tau}{2}\mathcal {A}\right)\varepsilon^{k}=\left(I-\frac{\tau}{2}\mathcal {A}\right)\varepsilon^{k-1},
\end{equation*}
 with $\varepsilon^k=\left(\epsilon^k_{1},\epsilon^k_{2},\cdots,\epsilon^k_{M-1},\epsilon^k_{\frac{1}{2}},\epsilon^k_{\frac{3}{2}},\cdots,\epsilon^k_{M-\frac{1}{2}}\right)^{T}$.
Upon relabeling and reorienting the  vectors $\varepsilon^k$ as
\begin{equation*}
\begin{split}
\widetilde{\varepsilon}^k&=\left(\epsilon^k_{\frac{1}{2}},\epsilon^k_{1},\epsilon^k_{\frac{3}{2}},\epsilon^k_{2},\cdots,\epsilon^k_{M-1},\epsilon^k_{M-\frac{1}{2}}\right)^{T},
\end{split}
\end{equation*}
then the above equation can be recast as
 \begin{equation*}
\left(I+\frac{\tau}{2}\widetilde{\mathcal {A}}\right)\widetilde{\varepsilon}^{k}=\left(I-\frac{\tau}{2}\widetilde{\mathcal {A}}\right)\widetilde{\varepsilon}^{k-1},
\end{equation*}
i.e.,
\begin{equation*}
\left(1+\frac{\tau}{2}a_{i,i}\right)\epsilon^k_{\frac{i}{2}}=\epsilon^{k-1}_\frac{i}{2}-\frac{\tau}{2}\mathop\sum\limits_{j=1}^{2M-1}a_{i,j}\epsilon^{k-1}_\frac{j}{2}
-\frac{\tau}{2}\sum^{2M-1}\limits_{j=1,j\neq i}a_{i,j}\epsilon^{k}_\frac{j}{2}~~{\rm with}~~\mathcal{\widetilde{A}}=\left\{ a_{i,j} \right\}_{i,j=1}^{2M-1}.
\end{equation*}

Let $\left|\epsilon^k_{\frac{i_{0}}{2}}\right|:=\left|\left|\varepsilon^k\right|\right|_{\infty}=\mathop{\max}\limits_{1\leq i\leq 2M-1}|\epsilon^k_\frac{i}{2}|$.
From Lemmas \ref{lemma3.3} and \ref{lemma3.4}, it yields $a_{i,j}<0$, $i\neq j$, and $\sum^{2M-1}\limits_{j=1,j\neq i}\left|a_{i,j}\right|<a_{i,i}$, $0<a_{i,i}< C_a:=\frac{2\left(b-a\right)^{1-\gamma}}{1-\gamma}$,then
\begin{equation*}
\begin{split}
\left(1+\frac{\tau}{2}a_{i_0,i_0}\right)\left|\left|\varepsilon^k\right|\right|_{\infty}
&\leq \left|\epsilon^{k-1}_\frac{i_0}{2}\right|+\frac{\tau}{2}\mathop\sum\limits_{j=1}^{2M-1}\left|a_{i_0,j}\right|\left|\epsilon^{k-1}_\frac{j}{2}\right|
+\frac{\tau}{2}\mathop\sum^{2M-1}\limits_{j=1,j\neq i_0}\left|a_{i_0,j}\right|\left|\epsilon^{k}_\frac{j}{2}\right|\\
&\leq \left|\left|\varepsilon^{k-1}\right|\right|_{\infty}+\frac{\tau}{2}\mathop\sum\limits_{j=1}^{2M-1}\left|a_{i_0,j}\right|\left|\left|\varepsilon^{k-1}\right|\right|_{\infty}
+\frac{\tau}{2}\mathop\sum^{2M-1}\limits_{j=1,j\neq i_0}\left|a_{i_0,j}\right|\left|\left|\varepsilon^{k}\right|\right|_{\infty},
\end{split}
\end{equation*}
it implies that
\begin{equation*}
\left|\left|\varepsilon^k\right|\right|_\infty
\leq \left(1+C_a\tau\right)\left|\left|\varepsilon^{k-1}\right|\right|_\infty\leq \left(1+C_a\tau\right)^k\left|\left|\varepsilon^{0}\right|\right|_\infty
\leq \exp(TC_a)\left|\left|\varepsilon^0\right|\right|_\infty.
\end{equation*}
The proof is completed.
\end{proof}

\begin{theorem}\label{theorem3.7}
Let $u\left(x_{i/2},t_k\right)$ be the exact solution of (\ref{1.1}) with $0< \gamma < 1$, and $u_{i/2}^k$ the  solution of
the numerical  scheme (\ref{2.4}). Then
\begin{equation*}
  \begin{split}
\left|\left|u(x_{i/2},t_k)-u_{i/2}^k\right|\right| =\mathcal {O}\left(\tau^2+h^{4-\gamma}\right),  \quad i=1,2,\ldots,2M-1;\,k=0,1,\ldots,N
  \end{split}
  \end{equation*}
with $N\tau\leq T.$
\end{theorem}
\begin{proof}
Denote $e_{i/2}^k=u(x_{i/2},t_k)-u_{i/2}^k$, $i=1,2,\ldots,2M-1;\,k=0,1,\ldots,N$.
 Subtracting (\ref{2.4}) from (\ref{1.1}) with $E^0=0$, it yields
\begin{equation*}
\left(I\!+\!\frac{\tau}{2}\mathcal {A}\right)E^{k}=\left(I\!-\!\frac{\tau}{2}\mathcal {A}\right)E^{k-1}\!+\!\tau R^{k-\frac{1}{2}},
\end{equation*}
with $E^k=\left(e^k_{1},e^k_{2},\cdots,e^k_{M-1},e^k_{\frac{1}{2}},e^k_{\frac{3}{2}},\cdots,e^k_{M-\frac{1}{2}}\right)^{T}$ and similarly for  $R^{k-\frac{1}{2}}$.
The local truncation error is  $R_i^{k-\frac{1}{2}}=\mathcal {O}\left(\tau^2+h^4\left(\eta_\frac{i}{2}\right)^{-\gamma}\right)\leq C_R\left(\tau^2+h^{4-\gamma}\right)$,
$\eta_\frac{i}{2}=\min\left\{x_{\frac{i}{2}}-a,b-x_{\frac{i}{2}}\right\}$ in \eqref{2.4} and $C_R$ is  a constant.

Upon relabeling and reorienting the  vectors $E^k$ and $R^{k-\frac{1}{2}}$ as
\begin{equation*}
\widetilde{E}^k\!=\!\left(e^k_{\frac{1}{2}},e^k_{1},e^k_{\frac{3}{2}},e^k_{2},\cdots,e^k_{M-1},e^k_{M-\frac{1}{2}}\right)^{T}\!,
\end{equation*}
\begin{equation*}
\widetilde{R}^{k\!-\!\frac{1}{2}}\!=\!\left(R^{k-\frac{1}{2}}_{\frac{1}{2}},R^{k-\frac{1}{2}}_{1},R^{k-\frac{1}{2}}_{\frac{3}{2}},R^{k-\frac{1}{2}}_{2},\cdots,R^{k-\frac{1}{2}}_{M-1},R_{M-\frac{1}{2}}\right)^{T},
\end{equation*}
then the above equation can be recast as
\begin{equation*}
\left(I+\frac{\tau}{2}\widetilde{\mathcal {A}}\right)\widetilde{E}^{k}=\left(I-\frac{\tau}{2}\widetilde{\mathcal {A}}\right)\widetilde{E}^{k-1}+\tau \widetilde{R}^{k-\frac{1}{2}},
\end{equation*}
i.e.,
\begin{equation*}
\left(1+\frac{\tau}{2}a_{i,i}\right)e^k_{\frac{i}{2}}=e^{k-1}_\frac{i}{2}-\frac{\tau}{2}\mathop\sum\limits_{j=1}^{2M-1}a_{i,j}e^{k-1}_\frac{j}{2}
-\frac{\tau}{2}\sum^{2M-1}\limits_{j=1,j\neq i}a_{i,j}e^{k}_\frac{j}{2}+\tau R_\frac{i}{2}^{k-\frac{1}{2}}~~{\rm with}~~\mathcal{\widetilde{A}}=\left\{ a_{i,j} \right\}_{i,j=1}^{2M-1}.
\end{equation*}
Let $\left|e^k_{\frac{i_{0}}{2}}\right|:=\left|\left|E^k\right|\right|_{\infty}=\mathop{\max}\limits_{1\leq i\leq 2M-1}\left|e^k_\frac{i}{2}\right|$.
Using  Lemmas \ref{lemma3.3} and \ref{lemma3.4}, we get  $a_{i,j}<0$, $i\neq j$, and $\sum^{2M-1}\limits_{j=1,j\neq i}\left|a_{i,j}\right|<a_{i,i}$, $0<a_{i,i}< C_a:=\frac{2\left(b-a\right)^{1-\gamma}}{1-\gamma}$. Therefore, we have
\begin{equation*}
\begin{split}
\left(1\!+\!\frac{\tau}{2}a_{i_0,i_0}\right)\left|\left|E^k\right|\right|_{\infty}
&\!\leq \!\left|e^{k-1}_\frac{i_0}{2}\right|+\frac{\tau}{2}\mathop\sum\limits_{j=1}^{2M-1}\left|a_{i_0,j}\right|\left|e^{k-1}_\frac{j}{2}\right|
+\frac{\tau}{2}\mathop\sum^{2M-1}\limits_{j=1,j\neq i_0}\left|a_{i_0,j}\right|\left|e^{k}_\frac{j}{2}\right|+\tau \left|R_\frac{i_0}{2}^{k-\frac{1}{2}}\right|\\
&\!\leq \!\left(\!\!1\!+\!\frac{\tau}{2}\mathop\sum\limits_{j=1}^{2M-1}\left|a_{i_0,j}\right|\right)\left|\left|E^{k-1}\right|\right|_{\infty}
\!\!+\!\frac{\tau}{2}\mathop\sum^{2M-1}\limits_{j=1,j\neq i_0}\left|a_{i_0,j}\right|\left|\left|E^{k}\right|\right|_{\infty}\!\!+\!\tau \left|R_\frac{i_0}{2}^{k-\frac{1}{2}}\right|.
\end{split}
\end{equation*}
It leads to
\begin{equation*}
\begin{split}
\left|\left|E^k\right|\right|_\infty
&\leq\left(1+\tau C_a\right)\left|\left|E^{k-1}\right|\right|_\infty+ C_R \left(\tau^2+h^{4-\gamma}\right)\tau\\
&\leq\left(1+\tau C_a\right)^k\left|\left|E^{0}\right|\right|_\infty+C_R \left(\tau^2+h^{4-\gamma}\right)\tau\mathop{\sum}\limits^{k-1}_{l=0}\left(1+\tau C_a\right)^l\\
&\leq C_R \left(\tau^2+h^{4-\gamma}\right)\tau\mathop{\sum}\limits^{k-1}_{l=0}\left(1+\tau C_a\right)^k
\leq C_R T \exp\left(TC_a\right)\left(\tau^2+h^{4-\gamma}\right).
\end{split}
\end{equation*}
The proof is completed.
\end{proof}

\subsection{Stability and convergence analysis for 2D with multiplicative Cauchy kernel}
First, we provide a local truncation error analysis, which is still lacking in \cite{CQSW:19}, for two-dimensional cases with multiplicative Cauchy kernel. Then the stability and convergence analysis are given.
\begin{lemma}\label{lemma3.8}
Let ${\mathcal{A}}=\mathcal{D}_x\otimes \mathcal{D}_y-\mathcal{G}_x\otimes \mathcal{G}_y$ be given in \eqref{2.11}. Then ${\mathcal{A}}$ is strictly diagonally dominant by rows.
\end{lemma}
\begin{proof}
From  Lemma \ref{lemma3.3}, we known  that $\mathcal{D}_x-\mathcal{G}_x$ and $\mathcal{D}_y-\mathcal{G}_y$ are strictly diagonally dominant by rows, and $\mathcal{G}_x$, $\mathcal{G}_y$ are  positive matrices.
Denote $\mathcal{D}_x=\{d^x_i\},\,\mathcal{G}_x=\{g^x_{i,l}\},\, i,l=1,2,\cdots,2M_x-1$, i.e.,
\begin{equation*}\begin{array}{cc}
\mathcal{D}_x=\left(\begin{array}{cccc}d^x_1 & ~ & ~ & ~\\
~ & d^x_2 & ~ & ~\\
~&~ & \ddots &~\\
~ & ~ & ~ & d^x_{2M_x-1}\end{array}\right)&
\mathcal{G}_x=\left(\begin{array}{cccc}g^x_{1,1} & g^x_{1,2} & \cdots & g^x_{1,2M_x-1}\\[2mm]
g^x_{2,1} & g^x_{2,2} & \cdots & g^x_{2,2M_x-1}\\[2mm]
\vdots&\vdots & \ddots &\vdots\\[2mm]
g^x_{2M_x-1,1} & g^x_{2M_x-1,2} & \cdots & g^x_{2M_x-1,2M_x-1}\end{array}\right)
\end{array}.
\end{equation*}
Similarly, we can denote $\mathcal{D}_y=\{d^y_j\},\,\mathcal{G}_y=\{g^y_{j,r}\},\, j,r=1,2,\cdots,2M_y-1$. Then

\begin{equation*}
\begin{split}
\mathcal{D}_x\otimes \mathcal{D}_y-\mathcal{G}_x\otimes \mathcal{G}_y=&\left(\begin{array}{cccc}d^x_1\otimes \mathcal{D}_y & ~ & ~ & ~\\
~ & d^x_2\otimes \mathcal{D}_y & ~ & ~\\
~&~ & \ddots &~\\
~ & ~ & ~ & d^x_{2M_x-1}\otimes \mathcal{D}_y\end{array}\right)\\
&-\left(\begin{array}{cccc}g^x_{1,1}\otimes \mathcal{G}_y & g^x_{1,1}\otimes \mathcal{G}_y & \cdots & g^x_{1,2M_x-1}\otimes \mathcal{G}_y\\[2mm]
g^x_{2,1}\otimes \mathcal{G}_y & g^x_{2,2}\otimes \mathcal{G}_y & \cdots & g^x_{2,2M_x-1}\otimes \mathcal{G}_y\\[2mm]
\vdots&\vdots & \ddots &\vdots\\[2mm]
g^x_{2M_x-1,1}\otimes \mathcal{G}_y & g^x_{2M_x-1,2}\otimes \mathcal{G}_y & \cdots & g^x_{2M_x-1,2M_x-1}\otimes \mathcal{G}_y\end{array}\right),
\end{split}
\end{equation*}
and the summation  by rows in ${\mathcal{A}}$ can be expressed as $\left\{d^x_id^y_j-\sum_{l=1}^{2M_x-1}g^x_{i,l}\sum_{r=1}^{2M_y-1}g^y_{j,r}\right\}$,  since $d^x_i>\sum_{l=1}^{2M_x-1}g^x_{i,l}$ and $d^y_j>\sum_{r=1}^{2M_x-1}g^y_{j,r}$.
The proof is completed.
\end{proof}

\begin{lemma}\label{lemma3.9}
Let $0<\gamma<1$, $u(x,y)\in C^6\left(\left[a,b\right]\times\left[c,d\right]\right)$ and $u_{Q}(x,y)$ be defined by \eqref{2.8}. Then for any $\left(x_{\frac{i}{2}},y_{\frac{j}{2}}\right)\in\left(a,b\right)\times\left(c,d\right)$, there exists
\begin{equation*}
\int^d_c\int^b_a\frac{u\left(x,y\right)-u_{Q}\left(x,y\right)}{\left|x_{\frac{i}{2}}-x\right|^\gamma\left|y_{\frac{j}{2}}-y\right|^\gamma}dxdy
=\mathcal{O}\left(h_x^4\left(\eta_{\frac{i}{2}}\right)^{-\gamma} +h_y^4\left(\widetilde{\eta}_{\frac{j}{2}}\right)^{-\gamma}\right)
+\mathcal{O}\left(h_x^{5-\gamma}+h_y^{5-\gamma}\right),
\end{equation*}
with $\eta_{\frac{i}{2}}=\min\left\{x_{\frac{i}{2}}-a,b-x_{\frac{i}{2}}\right\},\ \widetilde{\eta}_{\frac{j}{2}}=\min\left\{y_{\frac{j}{2}}-c,d-y_{\frac{j}{2}}\right\},\ i=1,2\ldots,2M_x-1,\ j=1,2\ldots,2M_y-1.$
\end{lemma}
\begin{proof}
Let  $\left(x,y\right)\in\left[x_{\frac{l}{2}},x_{\frac{l}{2}+1}\right]\times\left[y_{\frac{r}{2}},y_{\frac{r}{2}+1}\right]$,  $l=0,1,2\ldots,2M_x-2$, $r=0,1,2\ldots,2M_y-2$.
Using Taylor expansion and \eqref{2.8}, there exist $\xi_{\frac{l}{2}}\in \left[x_{\frac{l}{2}},x_{\frac{l}{2}+1}\right]$,   $\zeta_{\frac{r}{2}}\in \left[y_{\frac{r}{2}},y_{\frac{r}{2}+1}\right]$
and  $\varsigma_{\frac{r}{2}}\in \left[y_{\frac{r}{2}},y_{\frac{r}{2}+1}\right]$
such that
\begin{equation*}\label{03.1}
\begin{split}
 u\left(x,y\right)-u_{Q}\left(x,y\right)
&=u\left(x,y\right)-\sum^{2}_{s=0}\phi_{\frac{l+s}{2}}(x)u\left(x_{\frac{l+s}{2}},y\right)\\
&\qquad+\sum^{2}_{s=0}\phi_{\frac{l+s}{2}}(x)u\left(x_{\frac{l+s}{2}},y\right)
-\sum^{2}_{t=0}\phi_{\frac{r+t}{2}}(y)\sum^{2}_{s=0}\phi_{\frac{l+s}{2}}(x)u\left(x_{\frac{l+s}{2}},y_{\frac{r+t}{2}}\right)\\
&=:E_x+E_y+E_{xy},
\end{split}
\end{equation*}
where
\begin{equation*}
\begin{split}
E_x&=\frac{u^{(3)}_{x}\left(\xi_{\frac{l}{2}},y\right)}{3!}\left(x-x_{\frac{l}{2}}\right)\left(x-x_{\frac{l+1}{2}}\right)\left(x-x_{\frac{l}{2}+1}\right),\\
E_y&=\frac{u^{(3)}_{y}\left(x,\zeta_{\frac{r}{2}}\right)}{3!}\left(y-y_{\frac{r}{2}}\right)\left(y-y_{\frac{r+1}{2}}\right)\left(y-y_{\frac{r}{2}+1}\right),
\end{split}
\end{equation*}
and
\begin{equation*}
\begin{split}
E_{xy}=
\frac{-1}{(3!)^2}\frac{\partial^6 u \left(\varsigma_{\frac{r}{2}},\zeta_{\frac{r}{2}}\right)}{\partial x^3 \partial y^3}\left(x\!-x_{\frac{l}{2}}\right)\!
\left(x\!-x_{\frac{l+1}{2}}\right)\!\left(x\!-x_{\frac{l}{2}+1}\right)\!
\left(y\!-y_{\frac{r}{2}}\right)\!\left(y\!-y_{\frac{r+1}{2}}\right)\!\left(y\!-y_{\frac{r}{2}+1}\right).
\end{split}
\end{equation*}
Form  Theorem $3.7$ of \cite{CQSW:19} and  Lemma \ref{lemma3.4}, for any $\left(x_{\frac{i}{2}},y_{\frac{j}{2}}\right)\in\left(a,b\right)\times\left(c,d\right)$, $i=1,2\ldots,2M_x-1$, $j=1,2\ldots,2M_y-1$, we have
\begin{equation*}
\begin{split}
&\left|\!\int^d_c\!\int^b_a\frac{u\left(x,y\right)-u_{Q}\left(x,y\right)}{\left|x_{\frac{i}{2}}-x\right|^\gamma\left|y_{\frac{j}{2}}-y\right|^\gamma}dxdy\right|
=\left|\sum^{2M_y-2}_{r=0}\!\sum^{2M_x-2}_{l=0}\!\int^{y_{\frac{r}{2}+1}}_{y_{\frac{r}{2}}}\!\int^{x_{\frac{l}{2}+1}}_{x_{\frac{l}{2}}}\!\frac{E_x+E_y+E_{xy}}{\left|x_{\frac{i}{2}}-x\right|^\gamma\left|y_{\frac{j}{2}}-y\right|^\gamma}dxdy\right|\\
\leq&\left|\sum^{2M_x-2}_{l=0}\!\int^{x_{\frac{l}{2}+1}}_{x_{\frac{l}{2}}}\frac{E_x}{\left|x_{\frac{i}{2}}-x\right|^\gamma}dx\!\int^{d}_{c}\!\frac{1}{\left|y_{\frac{j}{2}}-y\right|^\gamma}dy\right|
\!+\!\left|\!\int^{b}_{a}\frac{1}{\left|x_{\frac{i}{2}}-x\right|^\gamma}dx\!\sum^{2M_y-2}_{r=0}\!\int^{y_{\frac{r}{2}+1}}_{y_{\frac{r}{2}}}\frac{E_y+E_{xy}}{\left|y_{\frac{j}{2}}-y\right|^\gamma}dy\right|\\
=&\mathcal{O}\left(h_x^4\left(\eta_{\frac{i}{2}}\right)^{-\gamma}\right)+\mathcal{O}\left(h_x^{5-\gamma}\right)
+\mathcal{O}\left(h_y^4\left(\widetilde{\eta}_{\frac{j}{2}}\right)^{-\gamma}\right)+\mathcal{O}\left(h_y^{5-\gamma}\right)+\mathcal{O}\left(h_x^3h_y^3\right)
\end{split}
\end{equation*}
with
$\eta_{\frac{i}{2}}=\min\left\{x_{\frac{i}{2}}-a,b-x_{\frac{i}{2}}\right\},\ \widetilde{\eta}_{\frac{j}{2}}=\min\left\{y_{\frac{j}{2}}-c,d-y_{\frac{j}{2}}\right\}$.
The proof is completed.
\end{proof}

\begin{theorem}\label{Theorem3.10}
The numerical schemes \eqref{2.13} is unconditionally stable.
\end{theorem}
\begin{proof}
Let $\widetilde{u}_{\frac{i}{2},\frac{j}{2}}^k~(i=1,2,\ldots,2M_x-1;j=1,2,\ldots,2M_y-1;\,k=0,1,\ldots,N)$ be the approximate solution of $u_{\frac{i}{2},\frac{j}{2}}^k$,
which is the exact solution of the difference scheme (\ref{2.13}).
Putting $\epsilon_{\frac{i}{2},\frac{j}{2}}^k=\widetilde{u}_{\frac{i}{2},\frac{j}{2}}^k-u_{\frac{i}{2},\frac{j}{2}}^k$, and we denote
\begin{equation*}
\varepsilon_i^k=\left(\epsilon_{i,1}^k,\epsilon_{i,2}^k,\cdots,\epsilon_{i,M_y-1}^k,\epsilon_{i,\frac{1}{2}}^k,\epsilon_{i,\frac{3}{2}}^k,\cdots,\epsilon_{i ,M_y-\frac{1}{2}}^k\right),
\ i=1,2,\cdots, M_x-1,\textstyle\frac{1}{2},\frac{3}{2},\cdots,M_x-\frac{1}{2}
\end{equation*}
then using (\ref{2.13}), we obtain the following perturbation equation
\begin{equation*}
\left(I+\frac{\tau}{2}{\mathcal{A}}\right){\varepsilon}^{k}=\left(I-\frac{\tau}{2}{\mathcal{A}}\right){\varepsilon}^{k-1},
\end{equation*}
 with ${\varepsilon}^k=\left(\varepsilon_1^k,\varepsilon_2^k,\cdots,\varepsilon_{M_x-1}^k,\varepsilon_{\frac{1}{2}}^k,\varepsilon_{\frac{3}{2}}^k,\cdots,\varepsilon_{M_x-\frac{1}{2}}^k \right)^T$.
Upon relabeling and reorienting the  vectors ${\varepsilon}^k$ as
\begin{equation*}
\widetilde{{\varepsilon}}^k=\left(\widetilde{\varepsilon}^k_{\frac{1}{2}},\widetilde{\varepsilon}^k_{1},\widetilde{\varepsilon}^k_{\frac{3}{2}},\widetilde{\varepsilon}^k_{2},\cdots,\widetilde{\varepsilon}^k_{M_x-1},\widetilde{\varepsilon}^k_{M_x-\frac{1}{2}}\right)^{T},
\end{equation*}
where
\begin{equation*}
\widetilde{\varepsilon}_i^k=\left(\epsilon_{i,\frac{1}{2}}^k,\epsilon_{i,1}^k,\epsilon_{i,\frac{3}{2}}^k,\epsilon_{i,2}^k,\cdots,\epsilon_{i,M_y-1}^k,\epsilon_{i,M_y-\frac{1}{2}}^k\right),
\ i=1,\frac{1}{2},2,\frac{3}{2},\cdots, M_x-1, M_x-\frac{1}{2},
\end{equation*}
then the above equation can be recast as
 \begin{equation*}
\left(I+\frac{\tau}{2}\widetilde{{\mathcal{A}}}\right)\widetilde{{\varepsilon}}^{k}=\left(I-\frac{\tau}{2}\widetilde{{\mathcal{A}}}\right)\widetilde{{\varepsilon}}^{k-1}
~~{\rm with}~~\widetilde{{\mathcal{A}}}=\widetilde{\mathcal{D}}_x\otimes \widetilde{\mathcal{D}}_y-\widetilde{\mathcal{G}}_x\otimes \widetilde{\mathcal{G}}_y,
\end{equation*}
where $\widetilde{\mathcal{D}}_x=\left\{ d^x_{i} \right\},\ \widetilde{\mathcal{G}}_x=\left\{ g^x_{i,l} \right\},\,\textstyle i,l =1,2,\ldots,{2M_x-1}$, $\widetilde{\mathcal{D}}_y=\left\{ d^y_{j} \right\},\ \widetilde{\mathcal{G}}_y=\left\{ g^y_{j,r} \right\},\,\textstyle j,r =1,2,\ldots,{2M_y-1}$, i.e.,
\begin{equation*}
\left(1\!+\frac{\tau}{2}d^x_{i}d^y_{j}\right)\epsilon_{\frac{i}{2},\frac{j}{2}}^k
-\frac{\tau}{2}\mathop\sum\limits_{l=1}^{2M_x-1}\!\mathop\sum\limits_{r=1}^{2M_y-1}\!g^x_{i,l}g^y_{j,r}\epsilon_{\frac{l}{2},\frac{r}{2}}^k
\!=\!\left(1\!-\frac{\tau}{2}d^x_{i}d^y_{j}\right)\epsilon_{\frac{i}{2},\frac{j}{2}}^{k-1}
+\frac{\tau}{2}\mathop\sum\limits_{l=1}^{2M_x-1}\!\mathop\sum\limits_{r=1}^{2M_y-1}\!g^x_{i,l}g^y_{j,r}\epsilon_{\frac{l}{2},\frac{r}{2}}^{k-1}.
\end{equation*}

Let $\left|\epsilon^k_{\frac{i_{0}}{2},\frac{j_{0}}{2}}\right|:=\left|\left|{\varepsilon}^k\right|\right|_{\infty}=\mathop{\max}\limits_{i,j}\left|\epsilon^k_{\frac{i}{2},\frac{j}{2}}\right|$.
From Lemmas \ref{lemma3.3} and \ref{lemma3.4}, it yields $0\!<d^x_{i}d^y_{j}\!< C_{d}\!:=\frac{4\left(b-a\right)^{1-\gamma}\left(d-c\right)^{1-\gamma}}{\left(1-\gamma\right)^2}$, using Theorem \ref{theorem3.6} and Lemma \ref{lemma3.8}, we have
\begin{equation*}
\begin{split}
&\left(1+\frac{\tau}{2}d^x_{i_0}d^y_{j_0}\right)\left|\left|{\varepsilon}^k\right|\right|_{\infty}\\
\leq&\left|\epsilon_{\frac{i_0}{2},\frac{j_0}{2}}^{k-1}\right|
\!+\!\frac{\tau}{2}d^x_{i_0}d^y_{j_0}\left|\epsilon_{\frac{i_0}{2},\frac{j_0}{2}}^{k-1}\right|
\!+\!\frac{\tau}{2}\mathop\sum\limits_{l=1}^{2M_x-1}\mathop\sum\limits_{r=1}^{2M_y-1}g^x_{i_0,l}g^y_{j_0,r}\left|\epsilon_{\frac{l}{2},\frac{r}{2}}^{k-1}\right|
\!+\!\frac{\tau}{2}\mathop\sum\limits_{l=1}^{2M_x-1}\mathop\sum\limits_{r=1}^{2M_y-1}g^x_{i_0,l}g^y_{j_0,r}\left|\epsilon_{\frac{l}{2},\frac{r}{2}}^k\right|\\
\!\leq&\left|\left|{\varepsilon}^{k-1}\right|\right|_{\infty}
\!+\!\frac{\tau}{2}\left(\!\!d^x_{i_0}d^y_{j_0}\!+\!\mathop\sum\limits_{l=1}^{2M_x-1}\mathop\sum\limits_{r=1}^{2M_y-1}g^x_{i_0,l}g^y_{j_0,r}\!\!\right)\left|\left|{\varepsilon}^{k-1}\right|\right|_{\infty}
\!+\!\frac{\tau}{2}\sum\limits_{l=1}^{2M_x-1}\sum\limits_{r=1}^{2M_y-1}g^x_{i_0,l}g^y_{j_0,r}\left|\left|{\varepsilon}^k\right|\right|_{\infty},
\end{split}
\end{equation*}
which leads to
\begin{equation*}
\left|\left|{\varepsilon}^k\right|\right|_{\infty}\leq \left(1+\tau C_{d} \right)\left|\left|{\varepsilon}^{k-1}\right|\right|_{\infty}
\leq\left(1+\tau C_{d} \right)^k\left|\left|{\varepsilon}^{0}\right|\right|_{\infty}\leq \exp\left(TC_{d}\right)\left|\left|{\varepsilon}^{0}\right|\right|_{\infty}.
\end{equation*}
The proof is completed.
\end{proof}

\begin{theorem}\label{Theorem3.11}
Let $u\left(x_{\frac{i}{2}},y_{\frac{j}{2}},t_k\right)$ be the exact solution of (\ref{2.12}) with $0< \gamma < 1$, and $u_{\frac{i}{2},\frac{j}{2}}^k$ the  solution of
the numerical  scheme (\ref{2.13}). Then
\begin{equation*}
  \begin{split}
\left|\left|u\left(x_{\frac{i}{2}},y_{\frac{j}{2}},t_k\right)-u_{\frac{i}{2},\frac{j}{2}}^k\right|\right| =\mathcal {O}\left(\tau^2+h_x^{4-\gamma}+h_y^{4-\gamma}\right),
  \end{split}
  \end{equation*}
with $i=1,2,\ldots,2M_x-1;j=1,2,\ldots,2M_y-1;\,k=0,1,\ldots,N$ and $N\tau\leq T$.
\end{theorem}
\begin{proof}
Denote $e_{\frac{i}{2},\frac{j}{2}}^k=u\left(x_{\frac{i}{2}},y_{\frac{j}{2}},t_k\right)-u_{\frac{i}{2},\frac{j}{2}}^k$, and
\begin{equation*}
E_i^k=\left(e_{i,1}^k,e_{i,2}^k,\cdots,e_{i,M_y-1}^k,e_{i,\frac{1}{2}}^k,e_{i,\frac{3}{2}}^k,\cdots,e_{i,M_y-\frac{1}{2}}^k\right),
\ i=1,2,\cdots, M_x-1,\textstyle\frac{1}{2},\frac{3}{2},\cdots,M_x-\frac{1}{2}
\end{equation*}
 Subtracting (\ref{2.13}) from (\ref{2.12}) with ${E}^0={0}$, a zero vector, it yields
\begin{equation*}
\left(I+\frac{\tau}{2}{\mathcal{A}}\right){E}^{k}=\left(I-\frac{\tau}{2}{\mathcal{A}}\right){E}^{k-1}+\tau {R}^{k-\frac{1}{2}},
\end{equation*}
with ${E}^k=\left(E^k_{1},E^k_{2},\cdots,E^k_{M_x-1},E^k_{\frac{1}{2}},E^k_{\frac{3}{2}},\cdots,E^k_{M_x-\frac{1}{2}}\right)^{T}$ and similarly for  ${R}^{k-\frac{1}{2}}$.
Here  the local truncation error is
${R}_{\frac{i}{2},\frac{j}{2}}^{k-\frac{1}{2}}
=\mathcal {O}\!\left(\!\tau^2+h^4\left(\eta_\frac{i}{2}\right)^{-\gamma}\!+\!h^4\left(\widetilde{\eta}_{\frac{j}{2}}\right)^{-\gamma}\!\right)
\leq C_R\left(\tau^2\!+\!h_x^{4-\gamma}+h_y^{4-\gamma}\right)$,
with $\eta_\frac{i}{2}=\min\left\{x_{\frac{i}{2}}-a,b-x_{\frac{i}{2}}\right\}, \widetilde{\eta}_{\frac{j}{2}}=\min\left\{y_{\frac{j}{2}}-c,d-y_{\frac{j}{2}}\right\}$ in \eqref{2.13} and $C_R$ is  a constant.

Upon relabeling and reorienting the  vectors ${E}^k$ and ${R}^{k-\frac{1}{2}}$ as $\widetilde{{E}}^k$ and $\widetilde{{R}}^{k-\frac{1}{2}}$ as in Theorem  \ref{Theorem3.10} again,
then the above equation can be recast as
\begin{equation*}
\left(I+\frac{\tau}{2}\widetilde{{\mathcal{A}}}\right)\widetilde{{E}}^{k}=\left(I-\frac{\tau}{2}\widetilde{{\mathcal{A}}}\right)\widetilde{{E}}^{k-1}+\tau \widetilde{{R}}^{k-\frac{1}{2}},
~~{\rm with}~~\widetilde{{\mathcal{A}}}=\widetilde{\mathcal{D}}_x\otimes \widetilde{\mathcal{D}}_y-\widetilde{\mathcal{G}}_x\otimes \widetilde{\mathcal{G}}_y,
\end{equation*}
where $\widetilde{\mathcal{D}}_x=\left\{ d^x_{i} \right\},\ \widetilde{\mathcal{G}}_x=\left\{ g^x_{i,l} \right\},\ \textstyle i,l =1,2,\ldots,{2M_x-1}$, $\widetilde{\mathcal{D}}_y=\left\{ d^y_{j} \right\},\ \widetilde{\mathcal{G}}_y=\left\{ g^y_{j,r} \right\},\ \textstyle j,r =1,2,\ldots,{2M_y-1}$, i.e.,
\begin{equation*}
\begin{split}
&\left(1+\frac{\tau}{2}d^x_{i}d^y_{j}\right)e_{\frac{i}{2},\frac{j}{2}}^k
-\frac{\tau}{2}\mathop\sum\limits_{l=1}^{2M_x-1}\mathop\sum\limits_{r=1}^{2M_y-1}g^x_{i,l}g^y_{j,r}e_{\frac{l}{2},\frac{r}{2}}^k\\
&\quad=\left(1-\frac{\tau}{2}d^x_{i}d^y_{j}\right)e_{\frac{i}{2},\frac{j}{2}}^{k-1}
+\frac{\tau}{2}\mathop\sum\limits_{l=1}^{2M_x-1}\mathop\sum\limits_{r=1}^{2M_y-1}g^x_{i,l}g^y_{j,r}e_{\frac{l}{2},\frac{r}{2}}^{k-1}
+\tau R_{\frac{i}{2},\frac{j}{2}}^{k-\frac{1}{2}}.
\end{split}
\end{equation*}
Let $\left|e^k_{\frac{i_{0}}{2},\frac{j_{0}}{2}}\right|:=\left|\left|{E}^k\right|\right|_{\infty}=\mathop{\max}\limits_{i,j}\left|e^k_{\frac{i}{2},\frac{j}{2}}\right|$.
Using  Lemmas \ref{lemma3.3} and \ref{lemma3.4}, we get  $0<d_id_j< C_{d}:=\frac{4\left(b-a\right)^{1-\gamma}\left(d-c\right)^{1-\gamma}}{(1-\gamma)^2}$, combine with Theorem \ref{theorem3.7} and Lemma \ref{lemma3.8}, we have
\small\begin{equation*}
\begin{split}
\left(1\!+\!\frac{\tau}{2}d^x_{i_0}d^y_{j_0}\right)\left|\left|{E}^k\right|\right|_{\infty}
\leq&\left|e_{\frac{i_0}{2},\frac{j_0}{2}}^{k-1}\right|
\!+\!\frac{\tau}{2}d^x_{i_0}d^y_{j_0}\left|e_{\frac{i_0}{2},\frac{j_0}{2}}^{k-1}\right|
\!+\!\frac{\tau}{2}\mathop\sum\limits_{l=1}^{2M_x-1}\mathop\sum\limits_{r=1}^{2M_y-1}g^x_{i_0,l}g^y_{j_0,r}\left|e_{\frac{l}{2},\frac{r}{2}}^{k-1}\right|\\
&\!+\!\frac{\tau}{2}\mathop\sum\limits_{l=1}^{2M_x-1}\mathop\sum\limits_{r=1}^{2M_y-1}g^x_{i_0,l}g^y_{j_0,r}\left|e_{\frac{l}{2},\frac{r}{2}}^k\right|
\!+\!\tau\left|R_{\frac{i_0}{2},\frac{j_0}{2}}^{k-\frac{1}{2}}\right|\\
\!\leq\!&\left(1\!+\!\tau C_{d}\right)\left|\left|{E}^{k-1}\right|\right|_\infty
\!\!+\!\frac{\tau}{2}\mathop\sum\limits_{l=1}^{2M_x-1}\mathop\sum\limits_{r=1}^{2M_y-1}\!\!g^x_{i_0,l}g^y_{j_0,r}\left|\left|{E}^k\right|\right|_{\infty}
\!\!\!+\!\tau\left|R_{\frac{i_0}{2},\frac{j_0}{2}}^{k-\frac{1}{2}}\right|,
\end{split}
\end{equation*}\small
which leads to
\begin{equation*}
\begin{split}
\left|\left|{E}^k\right|\right|_\infty
&\leq\left(1+\tau C_{d}\right)\left|\left|{E}^{k-1}\right|\right|_\infty+ C_R \left(\tau^2+h_x^{4-\gamma}+h_y^{4-\gamma}\right)\tau\\
&\leq\left(1+\tau C_{d}\right)^k\left|\left|{E}^{0}\right|\right|_\infty+C_R \left(\tau^2+h_x^{4-\gamma}+h_y^{4-\gamma}\right)\tau\mathop{\sum}\limits^{k-1}_{l=0}\left(1+\tau C_{d}\right)^l\\
&\leq C_R \left(\tau^2+h_x^{4-\gamma}+h_y^{4-\gamma}\right)\tau\mathop{\sum}\limits^{k-1}_{l=0}\left(1+\tau C_{d}\right)^k
\leq C_R T \exp\left(TC_{d}\right)\left(\tau^2+h_x^{4-\gamma}+h_y^{4-\gamma}\right).
\end{split}
\end{equation*}
The proof is completed.
\end{proof}

\subsection{Stability and convergence analysis for 2D with additive Cauchy kernels}
It should be noted that the stiffness matrix \eqref{2.11} with  multiplicative Cauchy kernels can be computed explicitly, but it is not for  additive Cauchy kernels.

\begin{lemma}\label{lemma3.12}
Let $0<\gamma<1$, $u(x,y)\in C^6\left(\left[a,b\right]\times\left[c,d\right]\right)$ and $u_{Q}(x,y)$ be defined by \eqref{2.8}. Then for any $\left(x_{\frac{i}{2}},y_{\frac{j}{2}}\right)\in\left(a,b\right)\times\left(c,d\right)$, there exists
\begin{equation*}
\int^d_c\int^b_a\frac{u\left(x,y\right)-u_{Q}\left(x,y\right)}{\left|\sqrt{\left(x_{\frac{i}{2}}-x\right)^2\!\!+\!\!\left(y_{\frac{j}{2}}-y\right)^2}\right|^{\gamma}}dxdy
=\mathcal{O}\left(h_x^4\left(\eta_{\frac{i}{2}}\right)^{-\gamma}\! +\!h_y^4\left(\widetilde{\eta}_{\frac{j}{2}}\right)^{-\gamma}\right)
+\mathcal{O}\left(h_x^{5-\gamma}\!+\!h_y^{5-\gamma}\right),
\end{equation*}
with $\eta_{\frac{i}{2}}=\min\left\{x_{\frac{i}{2}}-a,b-x_{\frac{i}{2}}\right\},\ \widetilde{\eta}_{\frac{j}{2}}=\min\left\{y_{\frac{j}{2}}-c,d-y_{\frac{j}{2}}\right\}i=1,2\ldots,2M_x-1,\ j=1,2\ldots,2M_y-1.$
\end{lemma}
\begin{proof}
According to  \eqref{03.1}, Lemma \ref{lemma3.4} and  Theorem $3.7$ in \cite{CQSW:19}, for any $\left(x_{\frac{i}{2}},y_{\frac{j}{2}}\right)\in\left(a,b\right)\times\left(c,d\right)$, $i=1,2\ldots,2M_x-1$, $j=1,2\ldots,2M_y-1$, we have
\small\begin{equation*}
\begin{split}
&\left|\int^d_c\int^b_a\frac{u\left(x,y\right)-u_{Q}\left(x,y\right)}{\left|\sqrt{\left(x_{i/2}-x\right)^2+\left(y_{j/2}-y\right)^2}\right|^{\gamma}}dxdy\right|\\
\!\leq&\left|\sum^{2M_x-2}_{l=0}\int^{x_{\frac{l}{2}+1}}_{x_{\frac{l}{2}}}\frac{E_x}{\left|x_{i/2}-x\right|^{\gamma}}dx
\int^{d}_{c}\frac{1}{\left|\sqrt{1\!+\left(\frac{y_{j/2}-y}{ x_{i/2}-x}\right)^2}\right|^{\gamma}}dy\right|\\
&\!+\left|\int^{b}_{a}\frac{1}{\left|\sqrt{\left(\frac{x_{i/2}-x}{y_{j/2}-y}\right)^2+1}\right|^{\gamma}}dx
\sum^{2M_y-2}_{r=0}\int^{y_{\frac{r}{2}\!+1}}_{y_{\frac{r}{2}}}\frac{E_y+E_{xy}}{|y_{j/2}-y|^\gamma}dy\right|\\
=&\mathcal{O}\left(h_x^4\left(\eta_{\frac{i}{2}}\right)^{-\gamma}\right)\!+\mathcal{O}\left(h_x^{5-\gamma}\right)
+\mathcal{O}\left(h_y^4\left(\widetilde{\eta}_{\frac{j}{2}}\right)^{-\gamma}\right)+\mathcal{O}\left(h_y^{5-\gamma}\right)+\mathcal{O}\left(h_x^3h_y^3\right)
\end{split}
\end{equation*}\small
with
$\eta_{\frac{i}{2}}=\min\left\{x_{\frac{i}{2}}-a,b-x_{\frac{i}{2}}\right\},\ \widetilde{\eta}_{\frac{j}{2}}=\min\left\{y_{\frac{j}{2}}-c,d-y_{\frac{j}{2}}\right\}$.
The proof is completed.
\end{proof}

\begin{theorem}\label{Theorem3.13}
The numerical schemes \eqref{2.20} is unconditionally stable.
\end{theorem}
\begin{proof}
Let $\widetilde{u}_{\frac{i}{2},\frac{j}{2}}^k~(i=1,2\ldots,2M_x-1;\,j=1,2\ldots,2M_y-1;\,k=0,1,\ldots,N)$ be the approximate solution of $u_{\frac{i}{2},\frac{j}{2}}^k$,
which is the exact solution of the difference scheme (\ref{2.20}).
Putting $\epsilon_{\frac{i}{2},\frac{j}{2}}^k=\widetilde{u}_{\frac{i}{2},\frac{j}{2}}^k-u_{\frac{i}{2},\frac{j}{2}}^k$, and we denote
\begin{equation*}
\varepsilon_i^k=\left(\epsilon_{i,1}^k,\epsilon_{i ,2}^k,\cdots,\epsilon_{i ,M_y-1}^k,\epsilon_{i,\frac{1}{2}}^k,\epsilon_{i,\frac{3}{2}}^k,\cdots,\epsilon_{i,M_y-\frac{1}{2}}^k\right),
\ i=1,2,\cdots, M_x-1,\textstyle\frac{1}{2},\frac{3}{2},\cdots,M_x-\frac{1}{2},
\end{equation*}
 then using (\ref{2.20}), we obtain the following perturbation equation
\begin{equation*}
\left(I+\frac{\tau}{2}{\mathcal{A}}\right){\varepsilon}^{k}=\left(I-\frac{\tau}{2}{\mathcal{A}}\right){\varepsilon}^{k-1},
\end{equation*}
 with ${\varepsilon}^k=\left(\varepsilon_1^k,\varepsilon_2^k,\cdots,\varepsilon_{M_x-1}^k,\varepsilon_{\frac{1}{2}}^k,\varepsilon_{\frac{3}{2}}^k,\cdots,\varepsilon_{M_x-\frac{1}{2}}^k \right)^T$.
Upon relabeling and reorienting the  vectors ${\varepsilon}^k$ as
\begin{equation*}
\widetilde{{\varepsilon}}^k=\left(\widetilde{\varepsilon}^k_{\frac{1}{2}},\widetilde{\varepsilon}^k_{1},\widetilde{\varepsilon}^k_{\frac{3}{2}},\widetilde{\varepsilon}^k_{2},\cdots,\widetilde{\varepsilon}^k_{M_x-1},\widetilde{\varepsilon}^k_{M_x-\frac{1}{2}}\right)^{T},
\end{equation*}
where
\begin{equation*}
\widetilde{\varepsilon}_i^k=\left(\epsilon_{i,\frac{1}{2}}^k,\epsilon_{i,1}^k,\epsilon_{i,\frac{3}{2}}^k,\epsilon_{i,2}^k,\cdots,\epsilon_{i,M_y-1}^k,\epsilon_{i,M_y-\frac{1}{2}}^k\right),
\ i=\textstyle\frac{1}{2},1,\textstyle\frac{3}{2},2,\cdots, M_x-1, \textstyle M_x-\frac{1}{2},
\end{equation*}
then the above equation can be recast as
 \begin{equation*}
\left(I+\frac{\tau}{2}\widetilde{{\mathcal{A}}}\right)\widetilde{{\varepsilon}}^{k}=\left(I-\frac{\tau}{2}\widetilde{{\mathcal{A}}}\right)\widetilde{{\varepsilon}}^{k-1},~~{\rm with}~~\widetilde{{\mathcal{A}}}=\widetilde{{\mathcal{D}}}-\widetilde{{\mathcal{G}}},
\end{equation*}
where $\widetilde{{\mathcal{D}}}=\left\{{ d}_{i,j} \right\}, \widetilde{{\mathcal{G}}}=\left\{ {g}_{i,j}^{l,r} \right\},\, i,l =1,2,\ldots,{2M_x-1};\, j,r =1,2,\ldots,{2M_y-1}$,
 i.e.,
\begin{equation*}
\left(1+\frac{\tau}{2}{d}_{i,j}\right)\epsilon_{\frac{i}{2},\frac{j}{2}}^k
-\frac{\tau}{2}\mathop\sum\limits_{l=1}^{2M_x-1}\mathop\sum\limits_{r=1}^{2M_y-1}{g}_{i,j}^{l,r}\epsilon_{\frac{l}{2},\frac{r}{2}}^k
=\left(1-\frac{\tau}{2}{d}_{i,j}\right)\epsilon_{\frac{i}{2},\frac{j}{2}}^{k-1}
+\frac{\tau}{2}\mathop\sum\limits_{l=1}^{2M_x-1}\mathop\sum\limits_{r=1}^{2M_y-1}{g}_{i,j}^{l.r}\epsilon_{\frac{l}{2},\frac{r}{2}}^{k-1}.
\end{equation*}
Let $\left|\epsilon^k_{\frac{i_{0}}{2},\frac{j_{0}}{2}}\right|:=\left|\left|{\varepsilon}^k\right|\right|_{\infty}=\mathop{\max}\limits_{i,j}\left|\epsilon^k_{\frac{i}{2},\frac{j}{2}}\right|$. Since
\begin{equation*}
\begin{split}
\int_a^b\int_c^d\!\frac{1}{\left|\sqrt{\left(x_{i/2}-\bar{x}\right)^2\!+\!\left(y_{j/2}-\bar{y}\right)^2}\right|^{\gamma}}d\bar{x}d\bar{y}
&\!=\!\int^{b}_{a}\int^{d}_{c}\!\frac{1}{\left|\sqrt{\left(\frac{x_{i/2}-\bar{x}}{y_{j/2}-\bar{y}}\right)^2\!+\!1}\right|^{\gamma}}\frac{1}{|y_{j/2}\!-\!\bar{y}|^\gamma}d\bar{x}d\bar{y}\\
&\!\leq\!\int^{b}_{a}\int^{d}_{c}\!\frac{1}{|y_{j/2}-\bar{y}|^\gamma}d\bar{x}d\bar{y} \leq C_d,
 \end{split}
 \end{equation*}
with $0\!<{d}_{i,j}\leq C_d=\frac{2(b-a)(d-c)^{1-\gamma}}{(1-\gamma)}$. From \eqref{2.6} and \eqref{2.7}, we have
 \begin{equation*}
 \begin{split}
 &\mathop\sum\limits_{l=1}^{2M_x-1}\mathop\sum\limits_{r=1}^{2M_y-1}\left|{g}_{i_0,j_0}^{l,r}\right|\\
 \leq &\mathop\sum\limits_{l=0}^{2M_x-1}\mathop\sum\limits_{r=0}^{2M_y-1}\left|\int^{x_{l+1}}_{x_l}\int^{y_{r+1}}_{y_r}
 \frac{\left(|\phi^+_{l}|+|\phi_{l+\frac{1}{2}}|+|\phi^-_{l+1}|\right)\left(|\phi^+_{r}|+|\phi_{r+\frac{1}{2}}|+|\phi^-_{r+1}|\right)}
 {\left|\sqrt{\left(x_{i/2}-\bar{x}\right)^2+\left(y_{j/2}-\bar{y}\right)^2}\right|^{\gamma}}d\bar{x}d\bar{y}\right|\\
 \leq&9 \mathop\sum\limits_{l=0}^{2M_x-1}\mathop\sum\limits_{r=0}^{2M_y-1}\left|\int^{x_{l+1}}_{x_l}\int^{y_{r+1}}_{y_r}
 \frac{1}{\left|\sqrt{\left(x_{i/2}-\bar{x}\right)^2+\left(y_{j/2}-\bar{y}\right)^2}\right|^{\gamma}}d\bar{x}d\bar{y}\right|\\
 \leq&9\int_a^b\int_c^d\frac{1}{\left|\sqrt{\left(x_{i/2}-\bar{x}\right)^2+\left(y_{j/2}-\bar{y}\right)^2}\right|^{\gamma}} d\bar{x}d\bar{y}\leq 9C_d.
 \end{split}
 \end{equation*}
Therefore
\begin{equation*}
\begin{split}
&\left(1+\frac{\tau}{2}{d}_{i_0,j_0}\right)\left|\left|{\varepsilon}^k\right|\right|_{\infty}\\
\leq&\left|\epsilon_{\frac{i_0}{2},\frac{j_0}{2}}^{k-1}\right|
+\frac{\tau}{2}{d}_{i_0,j_0}\left|\epsilon_{\frac{i_0}{2},\frac{j_0}{2}}^{k-1}\right|
+\frac{\tau}{2}\mathop\sum\limits_{l=1}^{2M_x-1}\mathop\sum\limits_{r=1}^{2M_y-1}\left|{g}_{i_0,j_0}^{l,r}\right|\left|\epsilon_{\frac{l}{2},\frac{r}{2}}^{k-1}\right|
+\frac{\tau}{2}\mathop\sum\limits_{l=1}^{2M_x-1}\mathop\sum\limits_{r=1}^{2M_y-1}\left|{g}_{i_0,j_0}^{l,r}\right|\left|\epsilon_{\frac{l}{2},\frac{r}{2}}^k\right|\\
\leq&\left(1+\frac{\tau}{2}{d}_{i_0,j_0}\right)\left|\left|{\varepsilon}^{k-1}\right|\right|_{\infty}
\!+\frac{\tau}{2}\mathop\sum\limits_{l=1}^{2M_x-1}\mathop\sum\limits_{r=1}^{2M_y-1}\left|{g}_{i_0,j_0}^{l,r}\right|\left|\left|{\varepsilon}^{k-1}\right|\right|_{\infty}
\!+\frac{\tau}{2}\mathop\sum\limits_{l=1}^{2M_x-1}\mathop\sum\limits_{r=1}^{2M_y-1}\left|{g}_{i_0,j_0}^{l,r}\right|\left|\left|{\varepsilon}^k\right|\right|_{\infty}\\
\leq& \left|\left|{\varepsilon}^{k-1}\right|\right|_{\infty}+5\tau C_d \left|\left|{\varepsilon}^{k-1}\right|\right|_{\infty}+\frac{9}{2}\tau  C_d \left|\left|{\varepsilon}^{k}\right|\right|_{\infty}
\end{split}
\end{equation*}
which leads to
\begin{equation*}
\left|\left|{\varepsilon}^k\right|\right|_{\infty}\leq \frac{1+5\tau C_d}{1-5\tau C_d}\left|\left|{\varepsilon}^{k-1}\right|\right|_\infty
\leq \left( \frac{1+5\tau C_d}{1-5\tau C_d}\right)^k\left|\left|{\varepsilon}^{0}\right|\right|_\infty
\leq \exp\left(\frac{10TC_d}{1-5\tau_0 C_d}\right)\left|\left|{\varepsilon}^0\right|\right|_\infty.
\end{equation*}
with $0<\tau<\tau_0=\frac{1}{10C_d}$. The proof is completed.
\end{proof}

\begin{theorem}\label{Theorem3.14}
Let $u\left(x_{\frac{i}{2}},y_{\frac{j}{2}},t_k\right)$ be the exact solution of (\ref{2.19}) with $0< \gamma < 1$, and $u_{\frac{i}{2},\frac{j}{2}}^k$ the  solution of
the numerical  scheme (\ref{2.20}). Then
\begin{equation*}
  \begin{split}
\left|\left|u\left(x_{\frac{i}{2}},y_{\frac{j}{2}},t_k\right)-u_{\frac{i}{2},\frac{j}{2}}^k\right|\right| =\mathcal {O}\left(\tau^2+h_x^{4-\gamma}+h_y^{4-\gamma}\right),
  \end{split}
  \end{equation*}
with $i=1,2,\ldots,2M_x-1;j=1,2,\ldots,2M_y-1;\,k=0,1,\ldots,N$ and $N\tau\leq T$.
\end{theorem}
\begin{proof}
Denote $e_{\frac{i}{2},\frac{j}{2}}^k=u\left(x_{\frac{i}{2}},y_{\frac{j}{2}},t_k\right)-u_{\frac{i}{2},\frac{j}{2}}^k$, and
\begin{equation*}
E_i^k=\left(e_{i,1}^k,e_{i,2}^k,\cdots,e_{i,M_y-1}^k,e_{i,\frac{1}{2}}^k,e_{i,\frac{3}{2}}^k,\cdots,e_{i,M_y-\frac{1}{2}}^k\right),
\ i=1,2,\cdots, M_x-1,\textstyle\frac{1}{2},\frac{3}{2},\cdots,M_x-\frac{1}{2}
\end{equation*}
 Subtracting (\ref{2.20}) from (\ref{2.19}) with ${E}^0={0}$, a zero vector, it yields
\begin{equation*}
\left(I+\frac{\tau}{2}{\mathcal{A}}\right){E}^{k}=\left(I-\frac{\tau}{2}{\mathcal{A}}\right){E}^{k-1}+\tau{ R}^{k-\frac{1}{2}},
\end{equation*}
with ${E}^k=\left(E^k_{1},E^k_{2},\cdots,E^k_{M_x-1},E^k_{\frac{1}{2}},E^k_{\frac{3}{2}},\cdots,E^k_{M_x-\frac{1}{2}}\right)^{T}$ and similarly for  ${R}^{k-\frac{1}{2}}$.
The local truncation error is
${R}_{\frac{i}{2},\frac{j}{2}}^{k-\frac{1}{2}}
=\mathcal {O}\left(\tau^2+h^4\left(\eta_\frac{i}{2}\right)^{-\gamma}+h^4\left(\widetilde{\eta}_{\frac{j}{2}}\right)^{-\gamma}\right)
\leq C_R\left(\tau^2+h_x^{4-\gamma}+h_y^{4-\gamma}\right)$,
with $\eta_\frac{i}{2}=\min\left\{x_{\frac{i}{2}}-a,b-x_{\frac{i}{2}}\right\}, \widetilde{\eta}_{\frac{j}{2}}=\min\left\{y_{\frac{j}{2}}-c,d-y_{\frac{j}{2}}\right\}$ in \eqref{2.20} and $C_R$ is  a constant.

Upon relabeling and reorienting the  vectors ${E}^k$ and ${R}^{k-\frac{1}{2}}$ as $\widetilde{{E}}^k$ and $\widetilde{{R}}^{k-\frac{1}{2}}$ as in Theorem  \ref{Theorem3.13} again,
then the above equation can be recast as
\begin{equation*}
\left(I+\frac{\tau}{2}\widetilde{{\mathcal{A}}}\right)\widetilde{{E}}^{k}=\left(I-\frac{\tau}{2}\widetilde{{\mathcal{A}}}\right)\widetilde{{E}}^{k-1}+\tau \widetilde{{R}}^{k-\frac{1}{2}},
~~{\rm with}~~\widetilde{{\mathcal{A}}}=\widetilde{{\mathcal{D}}}-\widetilde{{\mathcal{G}}},
\end{equation*}
where $\widetilde{{\mathcal{D}}}=\left\{{ d}_{i,j} \right\}, \widetilde{{\mathcal{G}}}=\left\{ {g}_{i,j}^{l,r} \right\},\, i,l =1,2,\ldots,{2M_x-1};\, j,r =1,2,\ldots,{2M_y-1}$,
 i.e.,
\begin{equation*}
\begin{split}
&\left(1+\frac{\tau}{2}{d}_{i,j}\right)e_{\frac{i}{2},\frac{j}{2}}^k
-\frac{\tau}{2}\mathop\sum\limits_{l=1}^{2M_x-1}\mathop\sum\limits_{r=1}^{2M_y-1}{g}_{i,j}^{l,r}e_{\frac{l}{2},\frac{r}{2}}^k\\
&\quad=\left(1-\frac{\tau}{2}{d}_{i,j}\right)e_{\frac{i}{2},\frac{j}{2}}^{k-1}
+\frac{\tau}{2}\mathop\sum\limits_{l=1}^{2M_x-1}\mathop\sum\limits_{r=1}^{2M_y-1}{g}_{i,j}^{l.r}e_{\frac{l}{2},\frac{r}{2}}^{k-1}
+\tau R_{\frac{i}{2},\frac{j}{2}}^{k-\frac{1}{2}}.
\end{split}
\end{equation*}
Let $\left|e^k_{\frac{i_{0}}{2},\frac{j_{0}}{2}}\right|:=||{E}^k||_{\infty}=\mathop{\max}\limits_{i,j}\left|e^k_{\frac{i}{2},\frac{j}{2}}\right|$
with  $0\!<{d}_{i,j}\leq C_d$ in Theorem \ref{Theorem3.13}. Then  we have
\small\begin{equation*}
\begin{split}
\left(1+\frac{\tau}{2}{d}_{i_0,j_0}\right)||{E}^k||_{\infty}
\leq&\left|e_{\frac{i_0}{2},\frac{j_0}{2}}^{k-1}\right|
+\frac{\tau}{2}{d}_{i_0,j_0}\left|e_{\frac{i_0}{2},\frac{j_0}{2}}^{k-1}\right|
+\frac{\tau}{2}\mathop\sum\limits_{l=1}^{2M_x-1}\mathop\sum\limits_{r=1}^{2M_y-1}{g}_{i_0,j_0}^{l,r}\left|e_{\frac{l}{2},\frac{r}{2}}^{k-1}\right|\\
&+\frac{\tau}{2}\mathop\sum\limits_{l=1}^{2M_x-1}\mathop\sum\limits_{r=1}^{2M_y-1}{g}_{i_0,j_0}^{l,r}\left|e_{\frac{l}{2},\frac{r}{2}}^k\right|
+\tau\left|R_{\frac{i_0}{2},\frac{j_0}{2}}^{k-\frac{1}{2}}\right|\\
\leq&||{E}^{k-1}||_{\infty}
+\frac{\tau}{2}{d}_{i_0,j_0}||{E}^{k-1}||_{\infty}
+\frac{\tau}{2}\mathop\sum\limits_{l=1}^{2M_x-1}\mathop\sum\limits_{r=1}^{2M_y-1}{g}_{i_0,j_0}^{l,r}||{E}^{k-1}||_{\infty}\\
&+\frac{\tau}{2}\mathop\sum\limits_{l=1}^{2M_x-1}\mathop\sum\limits_{r=1}^{2M_y-1}{g}_{i_0,j_0}^{l,r}||{E}^k||_{\infty}
+\tau\left|R_{\frac{i_0}{2},\frac{j_0}{2}}^{k-\frac{1}{2}}\right|\\
\leq& ||{E}^{k-1}||_{\infty}+5\tau C_d ||{E}^{k-1}||_{\infty}+\frac{9}{2}\tau  C_d ||{E}^{k}||_{\infty}+\tau\left|R_{\frac{i_0}{2},\frac{j_0}{2}}^{k-\frac{1}{2}}\right|,
\end{split}
\end{equation*}\small
which leads to
\begin{equation*}
\begin{split}
\left|\left|{E}^k\right|\right|_\infty
&\leq\left(\frac{1+5\tau C_d}{1-5\tau C_d}\right)\left|\left|{E}^{k-1}\right|\right|_\infty+ C_R \left(\tau^2+h_x^{4-\gamma}+h_y^{4-\gamma}\right)\tau\\
&\leq\left(\frac{1+5\tau C_d}{1-5\tau C_d}\right)^k\left|\left|{E}^{0}\right|\right|_\infty+C_R \left(\tau^2+h_x^{4-\gamma}+h_y^{4-\gamma}\right)\tau\mathop{\sum}\limits^{k-1}_{l=0}\left(\frac{1+5\tau C_d}{1-5\tau C_d}\right)^l\\
&\leq C_R \left(\tau^2+h_x^{4-\gamma}+h_y^{4-\gamma}\right)\tau\mathop{\sum}\limits^{k-1}_{l=0}\left(\frac{1+5\tau C_d}{1-5\tau C_d}\right)^k\\
&\leq C_R T \exp\left(\frac{10T C_d}{1-5\tau_0 C_d}\right)\left(\tau^2+h_x^{4-\gamma}+h_y^{4-\gamma}\right),
\end{split}
\end{equation*}
with $\ 0<\tau<\tau_0=\frac{1}{10C_d}$. The proof is completed.
\end{proof}

\section{Fast Conjugate Gradient Squared for nonsymmetric and indefinite linear systems}\label{Algorithms}
In this section, we develop  fast  Conjugate Gradient Squared algorithm to solve the resulting nonsymmetric and indefinite linear systems including rectangular matrices.
\subsection{The operation count and storage requirement}
To the best of our knowledge, most of the early works on fast Toeplitz solvers were focused on squared matrices by Fast fourier transform (FFT) \cite{Chan:07,ChanNg:96}.
Based on the idea of \cite{Chan:07,CWCD:14,CD:17,Pang:12,Wang:12}, we develop a fast algorithm for the rectangular matrices $\mathcal{P}$ and $\mathcal{Q}$, which  realizes  the computational count $\mathcal{O}(M \log  M)$ and the required storage $\mathcal{O}(M)$.
Let
\begin{equation*}
T_{M-1}=\left[
\begin{array}{ccccc}
t_0        & t_1     & t_2       & \cdots    & t_{M-2}\\
t_{-1}   &t_0      & t_1       &\ddots      &\vdots\\
t_{-2}   &t_{-1} &t_0        &\ddots       &t_2\\
\vdots   &\ddots  &\ddots   &\ddots   &t_1\\
t_{2-M} &\cdots  &t_{-2} &t_{-1}  &t_0\\
\end{array}\right].
\end{equation*}
Then, for any $\left(M-1\right)$-by-$1$   vector $\bf x$, the multiplication $T_{M-1}\bf x$ can also be computed by FFTs with  the computational count $\mathcal{O}(M \log  M)$  \cite[p.\,12]{Chan:07}. More concretely, we take a $2\left(M-1\right)$-by-$2\left(M-1\right)$
circulant matrix with $T_{M-1}$ embedded inside as follows:
\begin{equation*}
\left[\begin{array}{cc}T_{M-1} & \ast\\ \ast & T_{M-1}\end{array}\right]\left[\begin{array}{c}\bf x\\ \bf 0\end{array}\right]=\left[\begin{array}{c}T_{M-1}\bf x\\ \ddag\end{array}\right].
\end{equation*}
Therefore, we can develop this idea to compute the  rectangular matrices $\mathcal{P}_{M\times(M-1)}$ and $\mathcal{Q}_{(M-1)\times M}$.
More precisely,  we first embed $\mathcal{P}_{M\times(M-1)}$  of \eqref{2.3} into a $M$-by-$M$ Toeplitz matrix, i.e,
\begin{equation*}
\begin{split}
\mathcal{\widetilde{P}}=\left [ \begin{matrix}
p_{0}              & p_{1}              & p_{2}             &     \cdots & p_{M-3}   & p_{M-2}   &0          \\
p_{0}              & p_{0}              & p_{1}             &     \ddots & \ddots   & p_{M-3}   &p_{M-2}     \\
p_{1}              & p_{0}              & p_{0}             &     \ddots & \ddots   & \ddots   &p_{M-3}      \\
\vdots             & \ddots             &  \ddots           &     \ddots & \ddots    & p_{2}    & \vdots       \\
p_{M-4}            & \ddots            & \ddots           &     \ddots & p_{0}     & p_{1}     & p_{2}         \\
p_{M-3}            & p_{M-4}            & \ddots           &     p_{1} & p_{0}     & p_{0}     & p_{1}          \\
p_{M-2}            & p_{M-3}            & p_{M-4}           &     \cdots & p_{1}     & p_{0}     & p_{0}
 \end{matrix}
 \right ]_{M \times M}.
\end{split}
\end{equation*}
Then the multiplication $\mathcal{\widetilde{P}}\widetilde{\bf x}$ can also be computed by FFTs with  the computational count $\mathcal{O}(M \log M)$, i.e.,
\begin{equation*}
\left[\begin{array}{cc}\mathcal{\widetilde{P}} & \ast\\ \ast & \mathcal{\widetilde{P}}\end{array}\right]\left[\begin{array}{c}\widetilde{\bf x}\\ \bf 0\end{array}\right]
=\left[\begin{array}{c}\mathcal{\widetilde{P}}\widetilde{\bf x}\\ \ddag\end{array}\right]
=\left[\begin{array}{c}\mathcal{P}\bf x\\ \ddag\end{array}\right],~~\widetilde{x}=\left[\begin{array}{c}\bf x\\ 0\end{array}\right]_{M\times 1}.
\end{equation*}
On the other hand,  we embed $\mathcal{Q}_{(M-1)\times M}$ of \eqref{2.3} into the following  $M$-by-$M$ Toeplitz matrix,
\begin{equation*}
\begin{split}
\mathcal{\widetilde{Q}}=\left [ \begin{matrix}
q_{0}              & q_{0}              & q_{1}             &     \cdots & q_{M-4}   & q_{M-3}   & q_{M-2}      \\
q_{1}              & q_{0}              & q_{0}             &     \ddots & \ddots   & q_{M-4}   & q_{M-3}       \\
q_{2}              & q_{1}              & q_{0}             &     \ddots & \ddots   & \ddots   & q_{M-4}        \\
\vdots             & \ddots             &  \ddots           &     \ddots & \ddots    &  q_{1}    & \vdots           \\
q_{M-3}            & \ddots             & \ddots           &     \ddots & q_{0}     & q_{0}     &  q_{1}            \\
q_{M-2}            & q_{M-3}            & \ddots            & q_{2}  & q_{1}     & q_{0}     &  q_{0}            \\
0                  & q_{M-2}            & q_{M-3}           &     \cdots & q_{2}   & q_{1}     & q_{0}
 \end{matrix}
 \right ]_{M \times M}
\end{split}.
\end{equation*}
Hence the multiplication $\mathcal{\widetilde{Q}}\widetilde{\bf x}$ can also be computed by FFTs with  the computational count $\mathcal{O}(M \log M)$,
\begin{equation*}
\left[\begin{array}{cc}\mathcal{\widetilde{Q}} & \ast\\ \ast & \mathcal{\widetilde{Q}}\end{array}\right]\left[\begin{array}{c}\widetilde{\bf x}\\ \bf 0\end{array}\right]
=\left[\begin{array}{c}\mathcal{\widetilde{Q}}\widetilde{\bf x}\\ \ddag\end{array}\right],~~{\rm and}~~
\mathcal{\widetilde{Q}}\widetilde{\bf x}=\left[\begin{array}{c}\mathcal{Q}\bf \widetilde{x}\\ \dag\end{array}\right] ~~{\rm with }~~\dag \in \mathbb{R}.
\end{equation*}
Then, for  the matrix $\mathcal{A}$ of \eqref{2.3},  we only need to store   $4M$ parameters, instead of the full matrix $\mathcal{A}$ which has $4M^2$ parameters,
i.e., the required storage $\mathcal{O}(M)$. From fast Conjugate Gradient Squared Algorithm \ref{cgs1}  within finite iterations, see \cite{Saad:03,Sonneveld:89}, we have
the computational count $\mathcal{O}(M \log M)$. See Algorithm \ref{cgs1} in Appendix B. Two-dimensional  cases  can be similarly studied.

\subsection{Fast CGS  for nonsymmetric  indefinite linear systems  with rectangular matrices in 1D}
Let $U=\left[\begin{array}{l} w \\ v \end{array}\right]$ with $w=\left(u_{1},u_{2},\cdots,u_{M-1}\right)^{T}$,
and $v=\left(u_{\frac{1}{2}},u_{\frac{3}{2}},\cdots,u_{M-\frac{1}{2}}\right)^{T}$;
and similarly for  $F+K=\left[\begin{array}{l} F_w  \\ F_v  \end{array}\right]$. Then we can rewrite \eqref{2.3} as  the following  general linear system
\begin{equation*}
  \mathcal{A}\left[\begin{array}{l} w \\ v \end{array}\right]=\left[\begin{array}{l} F_w  \\ F_v  \end{array}\right],
\end{equation*}
and  employ the following fast Conjugate Gradient Squared Algorithm \ref{cgs1} to solve the steady-state nonlocal problems  (\ref{2.3});
and  Algorithms \ref{cgs1}-\ref{cgs2} in Appendix B to solve the time-dependent  nonlocal problems  (\ref{2.4}).

\subsection{Fast CGS  for 2D  nonlocal problems with  multiplicative Cauchy kernel }
From \eqref{2.10}, we have the grid functions
\begin{equation}\label{4.1}
\begin{split}
{U}
&=\left( U_1,U_2,\cdots,U_{M_x-1},U_{\frac{1}{2}},U_{\frac{3}{2}},\cdots,U_{M_x\!-\frac{1}{2}} \right)^T,\\
U_i
&=\left(u_{i,1},u_{i,2},\ldots,u_{i,M_y-1},u_{i,\frac{1}{2}},u_{i,\frac{3}{2}},\ldots,u_{i, M_y-\frac{1}{2}}\right),
\end{split}
\end{equation}
with $ i=1,2,\ldots, M_x-1,\textstyle\frac{1}{2},\frac{3}{2},\cdots,M_x-\frac{1}{2}$, then denote the matrix
 $${U}_{Mat}=\left( U_1^T,U_2^T,\cdots,U_{M_x-1}^T,U_{\frac{1}{2}}^T,U_{\frac{3}{2}}^T,\cdots,U_{M_x\!-\frac{1}{2}}^T \right).$$
Thus we first   employ the  fast Fourier transform transform   Algorithm \ref{mat-vec2} to compute the ${\mathcal{A}}{U}$ with ${\mathcal{A}}=\mathcal{D}_x\otimes \mathcal{D}_y - \mathcal{G}_x\otimes \mathcal{G}_y$ in  \eqref{2.11}.
Based on Algorithm \ref{mat-vec2},   we  use Algorithm \ref{cgs3} to solve the steady-state nonlocal problems  \eqref{2.11}
and  Algorithm \ref{cgs2} to solve the time-dependent  nonlocal problems  \eqref{2.13}.  See Algorithm \ref{cgs2}-\ref{cgs3} in Appendix B.

\subsection{Fast CGS  for 2D nonlocal problems with  additive Cauchy kernel nonlocal}
From \eqref{2.17}, we have  ${\mathcal{G}}=\left(\begin{array}{cc}{\mathcal{M}} & {\mathcal{Q}}\\ {\mathcal{P}} & {\mathcal{N}}\end{array}\right)$, and each block of ${\mathcal{M}}$ with $i,l=1,2,\cdots, M_x-1$ in the form of
\begin{equation*}
{\mathcal{M}}_{i,l}=\left(\begin{array}{cccc}
{\mathcal{M}}^{\mathcal{M}}_{i,l} & {\mathcal{M}}^{\mathcal{Q}}_{i,l} \\[2mm]
{\mathcal{M}}^{\mathcal{P}}_{i,l} & {\mathcal{M}}^{\mathcal{N}}_{i,l}
\end{array}\right),
\end{equation*}
see \eqref{2.17} and \eqref{2.18}. Similarly, we have ${\mathcal{Q}}_{i,l}$, ${\mathcal{P}}_{i,l}$ and ${\mathcal{N}}_{i,l}$ with different $i,\,l$, i.e.,
\begin{flalign*}
\begin{split}
\qquad\qquad &{\mathcal{Q}}_{i,l},\ {\rm for}\ i=1,2,\cdots, M_x-1, l=\textstyle\frac{1}{2},\frac{3}{2},\cdots,M_x-\frac{1}{2},\\[2mm]
\qquad\qquad &{\mathcal{P}}_{i,l},\ {\rm for}\ i=\textstyle\frac{1}{2},\frac{3}{2},\cdots,M_x-\frac{1}{2}, l=1,2,\cdots, M_x-1,\\[2mm]
\qquad\qquad &{\mathcal{N}}_{i,l},\ {\rm for}\ i=\textstyle\frac{1}{2},\frac{3}{2},\cdots,M_x-\frac{1}{2}, l=\textstyle\frac{1}{2},\frac{3}{2},\cdots,M_x-\frac{1}{2}.
\end{split}&&
\end{flalign*}
Let ${U}$ in \eqref{4.1} denote as ${U}=[{U}_L,{U}_R]^T$ with
$${U}_L=\left( U_1,U_2,\cdots,U_{M_x-1}\right)\ {\rm and }\ {U}_R=\left(U_{\frac{1}{2}},U_{\frac{3}{2}},\cdots,U_{M_x\!-\frac{1}{2}}\right).$$  Then
\begin{equation*}
{\mathcal{G}}{U}=\left(\begin{array}{c}
{\mathcal{M}}{U}_L^T + {\mathcal{Q}}{U}_R^T\\[2mm]
{\mathcal{P}}{U}_L^T + {\mathcal{N}}{U}_R^T
\end{array}\right).
\end{equation*}
Based on the block-Toeplitz-Toeplitz-block-like structural properties of ${\mathcal{M}}$, ${\mathcal{Q}}$, ${\mathcal{P}}$ and ${\mathcal{N}}$, we design fast algorithms for computing ${\mathcal{M}}{U}_L^T$ as an example.

First, to simplify the notation, let
 \begin{equation*}
{\mathcal{M}}_{i,l}=\left(\begin{array}{cccc}
{\mathcal{M}}^{\mathcal{M}}_{i,l} & {\mathcal{M}}^{\mathcal{Q}}_{i,l} \\[2mm]
{\mathcal{M}}^{\mathcal{P}}_{i,l} & {\mathcal{M}}^{\mathcal{N}}_{i,l}
\end{array}\right)
=\left(\begin{array}{cccc}
T_{\mathcal{M}} & T_{\mathcal{Q}} \\[2mm]
T_{\mathcal{P}} & T_{\mathcal{N}}
\end{array}\right),
\end{equation*}
where $T_{\mathcal{M}}$ and $ T_{\mathcal{N}}$ are squared Toeplitz matrix with the size of $\left(M_y-1\right)\times\left(M_y-1\right)$ and $M_y\times M_y$ respectively, $T_{\mathcal{Q}}$ with the size of $\left(M_y-1\right)\times M_y$ and $T_{\mathcal{P}}$ with the size of  $M_y\times\left(M_y-1\right)$ are rectangular ones. Then  embed $T_{\mathcal{M}}$, $ T_{\mathcal{Q}}$ and $T_{\mathcal{P}}$ into $M_y$-by-$M_y$ squared Toeplitz matrices and still denote $T_{\mathcal{M}}$, $ T_{\mathcal{Q}}$ and $T_{\mathcal{P}}$.
Next we embed the above  four $M_y$-by-$M_y$ Toeplitz matrices into a big circulant matrix, that is, construct a big circulant matrix ${{\mathcal{R}}}^{{\mathcal{M}}}_{i,l}$ with $T_{\mathcal{M}}$, $T_{\mathcal{Q}}$, $T_{\mathcal{P}}$ and $T_{\mathcal{N}}$ as follows:

\begin{equation}\label{4.2}
{{\mathcal{R}}}^{{\mathcal{M}}}_{i,l}=\left(\begin{array}{ccccccccc}
 S_1 & T_{\mathcal{M}} & S_2 & T_{\mathcal{Q}} & S_3 & T_{\mathcal{P}} & S_4 & T_{\mathcal{N}} & S_5 \\[2mm]
 S_5 & S_1 & T_{\mathcal{M}} & S_2 & T_{\mathcal{Q}} & S_3 & T_{\mathcal{P}} & S_4 & T_{\mathcal{N}} \\[2mm]
 T_{\mathcal{N}} & S_5 & S_1 & T_{\mathcal{M}} & S_2 & T_{\mathcal{Q}} & S_3 & T_{\mathcal{P}} & S_4 \\[2mm]
  S_4 & T_{\mathcal{N}} & S_5 & S_1 & T_{\mathcal{M}} & S_2 & T_{\mathcal{Q}} & S_3 & T_{\mathcal{P}} \\[2mm]
 T_{\mathcal{P}} &  S_4 & T_{\mathcal{N}} & S_5 & S_1 & T_{\mathcal{M}} & S_2 & T_{\mathcal{Q}} & S_3 \\[2mm]
  S_3 & T_{\mathcal{P}} &  S_4 & T_{\mathcal{N}} & S_5 & S_1 & T_{\mathcal{M}} & S_2 & T_{\mathcal{Q}} \\[2mm]
   T_{\mathcal{Q}} &  S_3 & T_{\mathcal{P}} &  S_4 & T_{\mathcal{N}} & S_5 & S_1 & T_{\mathcal{M}} & S_2 \\[2mm]
 S_2 &  T_{\mathcal{Q}} &  S_3 & T_{\mathcal{P}} &  S_4 & T_{\mathcal{N}} & S_5 & S_1 & T_{\mathcal{M}} \\[2mm]
T_{\mathcal{M}} &S_2 & T_{\mathcal{Q}} & S_3 & T_{\mathcal{P}} & S_4 & T_{\mathcal{N}} & S_5 & S_1
\end{array}\right),
\end{equation}
where $S_1$ with the size of  $M_y\times M_y$ is squared Toeplitz matrix, which  can be constructed by the partial entries of the first column of $T_{\mathcal{M}}$ and $0$ denotes the number zero, i.e.
\small\begin{equation*}
S_1\!=\!\left(\!\!\!\!\begin{array}{cccccc}
 0                       &T_{\mathcal{M}}(M_y,1)   & T_{\mathcal{M}}(M_y-1,1)  & \cdots   &T_{\mathcal{M}}(3,1)   & T_{\mathcal{M}}(2,1)\\[2mm]
T_{\mathcal{M}}(2,1)        &0                       & T_{\mathcal{M}}(M_y,1)     &\ddots    &\ddots                        &T_{\mathcal{M}}(3,1)\\[2mm]
T_{\mathcal{M}}(3,1)        & T_{\mathcal{M}}(2,1)      &0                          &\ddots   &\ddots                         &\vdots \\[2mm]
\vdots                             & \ddots                            &\ddots                                &\ddots   &\ddots                         &T_{\mathcal{M}}(M_y-1,1)\\[2mm]
T_{\mathcal{M}}(M_y-1,1) &\ddots                             &\ddots                               &\ddots   &0                    &T_{\mathcal{M}}(M_y,1) \\[2mm]
T_{\mathcal{M}}(M_y,1)    & T_{\mathcal{M}}(M_y-1,1)&\cdots                               &\cdots   &T_{\mathcal{M}}(2,1)   &0
\end{array}\!\!\!\!\right),
\end{equation*}\small
and $S_2$ with the size of $M_y\times M_y$ can be constructed by the partial entries of  the last column of $T_{\mathcal{M}}$ and the first column of $T_{\mathcal{Q}}$, for $S_2$ is between $T_{\mathcal{M}}$ and $T_{\mathcal{Q}}$ in ${{\mathcal{R}}}^{{\mathcal{M}}}_{i,l}$,  and $0$ denotes the number zero, i.e.,
\small\begin{equation*}
S_2\!=\!\left(\!\!\!\begin{array}{cccccc}
0                            & T_{\mathcal{Q}}(M_y,1)      & \cdots                             & \cdots        &T_{\mathcal{Q}}(3,1)        &T_{\mathcal{Q}}(2,1) \\[2mm]
T_{\mathcal{M}}(1,M_y)        &0                           & T_{\mathcal{Q}}(M_y,1)    &\ddots         &\ddots                             &T_{\mathcal{Q}}(3,1) \\[2mm]
T_{\mathcal{M}}(2,M_y)        &T_{\mathcal{M}}(1,M_y)       & 0                      &\ddots         &\ddots                             &\vdots \\[2mm]
\vdots                                 & \ddots                                &\ddots                               &\ddots        &\ddots                             & \vdots \\[2mm]
T_{\mathcal{M}}(M_y\!-\!2,M_y) &\ddots                                 & \ddots                             &\ddots         &0                       & T_{\mathcal{Q}}(M_y,1) \\[2mm]
T_{\mathcal{M}}(M_y\!-\!1,M_y) &T_{\mathcal{M}}(M_y\!-\!2,M_y) & \cdots                             & \cdots        & T_{\mathcal{M}}(1,M_y)  &0
\end{array}\!\!\!\!\right),
\end{equation*}\small
and similarly, we can denote the $M_y$-by-$M_y$ Toeplitz matrices $S_3$, $S_4$ and $S_5$ as $S_2$, for they are between $T_{\mathcal{Q}}$, $T_{\mathcal{P}}$ and $T_{\mathcal{N}}$.  Since ${{\mathcal{R}}}^{{\mathcal{M}}}_{i,l}$ is the Toeplitz matrix with the size of $9M_y\times 9M_y$, that means the vector $U_l$ $(l=1,2,\cdots,M_x-1)$ in ${U}_L$ are also regularly expanded into a vector with the length of $9M_y$.

 Then we can construct circulant matrices for  each block of ${\mathcal{M}}$, and denote ${{\mathcal{R}}}^{{\mathcal{M}}}_{i,l}$ as the resulting circulant matrix of ${\mathcal{M}}_{i,l}$. Thus, from ${\mathcal{M}}$ in \eqref{2.17}, by replacing ${\mathcal{M}}_{i,l}$ with ${{\mathcal{R}}}^{{\mathcal{M}}}_{i,l}$, we have the following block-Toeplitz-circulant-block (BTCB) matrix \cite{ChanNg:96,DuW:15}
\begin{equation*}
{{\mathcal{R}}}_{{\mathcal{M}}}=\left(\begin{array}{ccccc}
{{\mathcal{R}}}^{{\mathcal{M}}}_{1,1} & {{\mathcal{R}}}^{{\mathcal{M}}}_{1,2} & {{\mathcal{R}}}^{{\mathcal{M}}}_{1,3} & \cdots                                   &{{\mathcal{R}}}^{{\mathcal{M}}}_{1,M_x-1}\\[2mm]
{{\mathcal{R}}}^{{\mathcal{M}}}_{1,2} & {{\mathcal{R}}}^{{\mathcal{M}}}_{1,1} & {{\mathcal{R}}}^{{\mathcal{M}}}_{1,2} & \ddots                                   &\vdots\\[2mm]
{{\mathcal{R}}}^{{\mathcal{M}}}_{1,3} & {{\mathcal{R}}}^{{\mathcal{M}}}_{1,2} & {{\mathcal{R}}}^{{\mathcal{M}}}_{1,1} & \ddots                                   &{{\mathcal{R}}}^{{\mathcal{M}}}_{1,3}\\[2mm]
\vdots                                   & \ddots                                    & \ddots                                   & \ddots                                   & {{\mathcal{R}}}^{{\mathcal{M}}}_{1,2}\\[2mm]
{{\mathcal{R}}}^{{\mathcal{M}}}_{1,M_x-1} & \cdots                            & {{\mathcal{R}}}^{{\mathcal{M}}}_{1,3}  &{{\mathcal{R}}}^{{\mathcal{M}}}_{1,2} &{{\mathcal{R}}}^{{\mathcal{M}}}_{1,1}
\end{array}\right).
\end{equation*}
Similarly, we can obtain the BTCB matrix ${{\mathcal{R}}}_{{\mathcal{Q}}}$, ${{\mathcal{R}}}_{{\mathcal{P}}}$ and ${{\mathcal{R}}}_{{\mathcal{N}}}$ from ${\mathcal{Q}}$, ${\mathcal{P}}$ and ${\mathcal{N}}$ in \eqref{2.17}. The matrix ${{\mathcal{R}}}_{{\mathcal{M}}} $ can also be embedded into a BCCB matrix $\textbf{C}_{{\mathcal{M}}}$ with the size of $18M_y\left(M_x-1\right)\times 18M_y\left(M_x-1\right)$ as follows:
  \begin{equation*}
\textbf{C}_{{\mathcal{M}}}=\left(\begin{array}{cc}
 {{\mathcal{R}}}_{{\mathcal{M}}} & \widetilde{{{\mathcal{R}}}_{{\mathcal{M}}}}\\[2mm]
 \widetilde{{{\mathcal{R}}}_{{\mathcal{M}}}} & {{\mathcal{R}}}_{{\mathcal{M}}}
 \end{array}\right),
 \end{equation*}
 where
 \begin{equation*}
 \widetilde{{{\mathcal{R}}}_{{\mathcal{M}}}}=\left(\begin{array}{ccccc}
  \textbf{0}                                          &  {{\mathcal{R}}}^{{\mathcal{M}}}_{1,M_x-1} & {{\mathcal{R}}}^{{\mathcal{M}}}_{1,M_x-2}    &\cdots   &{{\mathcal{R}}}^{{\mathcal{M}}}_{1,2}\\[2mm]
{{\mathcal{R}}}^{{\mathcal{M}}}_{1,M_x-1}    & \textbf{0 }                                       &  {{\mathcal{R}}}^{{\mathcal{M}}}_{1,M_x-1}    &\ddots   &\vdots \\[2mm]
{{\mathcal{R}}}^{{\mathcal{M}}}_{1,M_x-2}    &{{\mathcal{R}}}^{{\mathcal{M}}}_{1,M_x-1}    &   \textbf{0 }                                        &\ddots      &{{\mathcal{R}}}^{{\mathcal{M}}}_{1,M_x-2}\\[2mm]
 \vdots                                  &\ddots                                    &\ddots                                       &\ddots      &{{\mathcal{R}}}^{{\mathcal{M}}}_{1,M_x-1}\\[2mm]
{{\mathcal{R}}}^{{\mathcal{M}}}_{1,2}         &\cdots           &{{\mathcal{R}}}^{{\mathcal{M}}}_{1,M_x-2}              &{{\mathcal{R}}}^{{\mathcal{M}}}_{1,M_x-1}      &\textbf{0}
  \end{array}\right),
  \end{equation*}
 with $\textbf{0}$ denotes the $9M_y$-by-$9M_y$ zero matrix. Let $\textbf{c}$ be the first column vector of the matrix $\textbf{C}_{{\mathcal{M}}}$. Let $F_{2\left(M_x-1\right)}\otimes F_{9M_y}$ be the two-dimensional discrete Fourier transform matrix. Then the matrix $\textbf{C}_{{\mathcal{M}}}$ has the following  diagonalization
  \begin{equation*}
  \begin{split}
  \textbf{C}_{{\mathcal{M}}}
  =\left(F_{2\left(M_x-1\right)}\otimes F_{9M_y}\right)^{-1}diag\left(\left(F_{2\left(M_x-1\right)}\otimes F_{9M_y}\right)\textbf{c}\right)\left(F_{2\left(M_x-1\right)}\otimes F_{9M_y}\right).
  \end{split}
   \end{equation*}
   That means we can compute ${\mathcal{M}}{U}_L^T$ by two-dimensional FFT, i.e., computing with the order fft2 and ifft2 by MATLAB. The algorithm also can be used to compute ${\mathcal{Q}}{U}_R^T$, ${\mathcal{P}}{U}_L^T$ and ${\mathcal{N}}{U}_R^T$ fast and efficiently.
Based on algorithm above,   we  use Algorithm \ref{cgs3}
 to solve the steady-state nonlocal problems  \eqref{2.17}
and  Algorithm \ref{cgs2} to solve the time-dependent  nonlocal problems  \eqref{2.20}.

\begin{remark}
To compute two-dimensional matrix-vector multiplication  $\left(G_x\otimes G_y\right){U}$ with multiplicative Cauchy kernel,  by Algorithm \ref{mat-vec2}, we can  reduce  the computational complexity  $\mathcal{O}\left(M_xM_y\log M_xM_y\right)$ to $\mathcal{O}\left(M_xM_y \left(\log M_x+\log M_y\right)\right)$.
Moreover, for   block-Toeplitz Toeplitz-block-like algebraic system in \eqref{2.17}, it only  needs
  $\mathcal{O}\left(M_xM_y\log M_xM_y\right)$.
\end{remark}

\section{Numerical results}
In this section, we numerically verify the above theoretical results including convergence rates and numerical stability. And the $l_\infty$ norm is used to measure the numerical errors.
 \subsection{Numerical results for 1D}
Consider one-dimensional time-dependent  nonlocal problem of \eqref{1.1}
with a finite domain $0=a<x<b=1$ and $t\in (0,1]$.
The source function is   easy to explicitly compute.

\begin{table}[!th]
\vspace{-2mm}
\renewcommand{\captionfont}{\footnotesize}
\centering
\caption{FCGS to solve  Crank-Nicolson  scheme in \eqref{2.4} with $\tau=h=(b-a)/M$. The exact solution is  $u(x,t)=e^t\left(x^2(b-x)^2+e^{-2}\right)$}
\resizebox{\textwidth}{!}{
\begin{tabular}{ccccc|cccc|cccc}
\hline\noalign{\smallskip}
\multirow{2}*{$M$}&  \multicolumn{4}{c}{$\gamma=0.2$}          &\multicolumn{4}{c}{$\gamma=0.5$}                &\multicolumn{4}{c}{$\gamma=0.8$}\\
                          \cline{2-13}\noalign{\smallskip}
                  &Error               &Rate         &CPU        &Iter           &Error              &Rate          &CPU         &Iter    &Error              &Rate          &CPU         &Iter\\ \hline
 $2^{7}$     &1.1223e-06      &~~~~      &0.3812s   &3              &1.1728e-06     &~~~~       &0.3610s     &3       &1.2235e-06     &~~~~       &0.4051s     &4\\
 $2^{8}$     &2.7995e-07      &2.0033     &0.8702s    &3              &2.9229e-07      &2.0045       &0.8240s    &3      &3.0432e-07      &2.0074       &0.8218s    &3\\
 $2^{9}$     &6.9907e-08      &2.0017     &2.1701s    &3              &7.2958e-08      &2.0023       &2.1487s    &3      &7.5887e-08      &2.0037       &2.1628s    &3 \\
 $2^{10}$   &1.7467e-08      &2.0008     &4.4068s    &2              &1.8225e-08      &2.0012       &5.0112s    &3       &1.8964e-08      &2.0006       &5.0554s    &3 \\\hline
\end{tabular}
}
\label{TT01}
\end{table}

\begin{table}[!th]
\vspace{-2mm}
\renewcommand{\captionfont}{\footnotesize}
\centering
\caption{FCGS to solve BDF4 scheme in \eqref{2.21} with $\tau=h=(b-a)/M$. The exact solution is  $u(x,t)=e^t\left(x^2(b-x)^2+e^{-2}\right)$}
\resizebox{\textwidth}{!}{
\begin{tabular}{ccccc|cccc|cccc}\hline\noalign{\smallskip}
\multirow{2}*{$M$}&  \multicolumn{4}{c}{$\gamma=0.2$}      &\multicolumn{4}{c}{$\gamma=0.5$}     &\multicolumn{4}{c}{$\gamma=0.8$} \\
                          \cline{2-13}\noalign{\smallskip}
                 &Error             &Rate          &CPU         &Iter       &Error           &Rate       &CPU      &Iter    &Error           &Rate       &CPU      &Iter \\\hline
  $2^5$      &6.9518e-08    &~~~~       &0.1657s    &3          &1.2045e-07  &~~~~    &0.1723   &4       &2.3806e-07  &~~~~    &0.1875s   &4     \\
  $2^6$      &4.9176e-09    &3.8214       &0.3678s    &3          &1.0789e-08  &3.5232    &0.3752   &3       &2.6632e-08  &3.1601    &0.3767s   &4      \\
  $2^7$      &3.4026e-10    &3.8533       &0.9078s    &3          &9.2911e-10  &3.5376    &0.8966   &3       &2.8910e-09  &3.2035    &0.9260s   &4        \\
  $2^8$      &2.3611e-11    &3.8491       &2.2222s    &3          &7.9967e-11  &3.5384    &2.2995   &3       &3.0890e-10  &3.2263    &2.2558s   &3          \\\hline
\end{tabular}
}
\label{TT02}
\end{table}

Tables  \ref{TT01} and \ref{TT02} show that   Crank-Niclson  scheme in \eqref{2.4} has the global  convergence rate  $\mathcal {O}\left(\tau^2+h^{4-\gamma}\right)$ and the computational cost is of
 $\mathcal{O}\left(M\log(M)\right)$ operations.
 \subsection{Numerical results for 2D with multiplicative Cauchy kernel}
Consider two-dimensional  nonlocal problem \eqref{2.12}
with a finite domain $0=a<x,y<b=2$ and $t\in (0,2]$. The source function is easy to explicitly compute.

\begin{table}[!th]
\vspace{-2mm}
\renewcommand{\captionfont}{\footnotesize}
\centering
\caption{FCGS to solve  Crank-Nicolson  scheme in \eqref{2.13} with $\tau=h_x=(b-a)/M_x$, $M_y=M_x$.
The exact solution is
$u(x,y,t)=e^t\left(x^2(b-x)^2y^2(b-y)^2-\sin(1)\right)$}
\resizebox{\textwidth}{!}{
\begin{tabular}{ccccc|cccc|cccc}
\hline\noalign{\smallskip}
\multirow{2}*{$M_x$}&  \multicolumn{4}{c}{$\gamma=0.2$}        &\multicolumn{4}{c}{$\gamma=0.5$}                   &\multicolumn{4}{c}{$\gamma=0.8$}\\
                          \cline{2-13}\noalign{\smallskip}
                  &Error               &Rate         &CPU         &Iter        &Error              &Rate          &CPU          &Iter      &Error                &Rate          &CPU          &Iter \\ \hline
 $2^{3}$     &2.1016e-02      &~~~~      &0.1530s    &8           &2.1562e-02      &~~~~        &0.2281s     &12      &2.2528e-02       &~~~~        &0.2887s     &17    \\
 $2^{4}$     &5.5269e-03      &1.9269     &0.5458s     &7           &5.6003e-03      &1.9449       &0.7042s      &9       &5.7243e-03       &1.9765       &1.3012s      &15     \\
 $2^{5}$     &1.3985e-03      &1.9826     &2.4249s     &6           &1.4106e-03      &1.9892       &2.9071s      &7       &1.4242e-03       &2.0069       &3.9693s      &12      \\
 $2^{6}$   &3.5060e-04      &1.9959     &11.4087s    &5           &3.5334e-04      &1.9971       &13.0410s    &6       &3.5620e-04       &1.9994       &17.6337s     &9         \\\hline
\end{tabular}
}
\label{TT03}
\end{table}

\begin{table}[!th]
\vspace{-2mm}
\renewcommand{\captionfont}{\footnotesize}
\centering
\caption{FCGS to solve BDF4 scheme in \eqref{2.21} with $\tau=h_x=(b-a)/M_x$, $M_y=M_x$.
The exact solution is
$u(x,y,t)=e^t\left(x^2(b-x)^2y^2(b-y)^2-\sin(1)\right)$}
\resizebox{\textwidth}{!}{
\begin{tabular}{ccccc|cccc|cccc}\hline\noalign{\smallskip}
\multirow{2}*{$M_x$}&  \multicolumn{4}{c}{$\gamma=0.2$}      &\multicolumn{4}{c}{$\gamma=0.5$}          &\multicolumn{4}{c}{$\gamma=0.8$}\\
                          \cline{2-13}\noalign{\smallskip}
                 &Error             &Rate          &CPU         &Iter       &Error           &Rate       &CPU        &Iter       &Error           &Rate       &CPU        &Iter \\\hline
  $2^3$      &3.0818e-03    &~~~~       &0.3900s    &6         &2.6856e-03  &~~~~    &0.4190s    &8        &2.9844e-03  &~~~~    &0.5832s    &12   \\
  $2^4$      &2.8489e-04    &3.4353       &1.4615s    &5          &2.7296e-04  &3.2985    &1.5581s    &6         &2.1386e-04  &3.8027    &2.1692s    &11    \\
  $2^5$      &2.1244e-05    &3.7453       &6.7009s    &5         &2.2197e-05  &3.6203    &6.6590s    &5         &2.0220e-05  &3.4028    &8.8939s     &9      \\
  $2^6$      &1.4568e-06    &3.8662       &32.1854s    &4         &1.6769e-06  &3.7526   &32.9718s   &5         &1.9644e-06  &3.3636    &38.4941    &6          \\\hline
\end{tabular}
}
\label{TT04}
\end{table}

Tables  \ref{TT03} and \ref{TT04} show that  Crank-Niclson  scheme in \eqref{2.13} has the global  convergence rate
 $\mathcal {O}\left(\tau^2+h_x^{4-\gamma}+h_y^{4-\gamma}\right)$ and the computational complexity is
 $\mathcal{O}\left(M_xM_y \left(\log M_x+\log M_y\right)\right)$.
 \subsection{Numerical results for 2D with additive Cauchy kernel}
Consider 2D time-dependent nonlocal problem  \eqref{2.19}
with a finite domain $0=a<x,y<b=1$. The source function is computed by Gauss quadrature.
\begin{table}[!th]
\vspace{-2mm}
\renewcommand{\captionfont}{\footnotesize}
\centering
\caption{FCGS to solve  Crank-Nicolson  scheme in \eqref{2.20} with  $\tau=1/1000$ and $h_x=(b-a)/M_x$, $M_y=M_x$¡£
The exact solution is  $u(x,y,t)=e^t\left( e^{(2x+4y)}\left(\sin(2x)+\cos(4y)\right)+1\right)$}
\resizebox{\textwidth}{!}{
\begin{tabular}{ccccc|cccc|cccc}
\hline\noalign{\smallskip}
\multirow{2}*{$M_x$}&  \multicolumn{4}{c}{$\gamma=0.2$}        &\multicolumn{4}{c}{$\gamma=0.5$}                   &\multicolumn{4}{c}{$\gamma=0.8$}\\
                          \cline{2-13}\noalign{\smallskip}
                  &Error               &Rate         &CPU         &Iter        &Error               &Rate           &CPU          &Iter      &Error             &Rate          &CPU          &Iter \\ \hline
 $2^{1}$     &1.0639e-01      &~~~~      &0.3172s    &3           &1.6147e-01      &~~~~        &0.3241s     &3      &2.5036e-01       &~~~~        &0.3181s     &3    \\
 $2^{2}$     &7.7522e-03      &3.7786     &1.9283s     &3           &1.3699e-02      &3.5592       &1.9610s      &3       &2.4766e-02       &3.3376       &2.0489s      &3    \\
 $2^{3}$     &5.5544e-04      &3.8029     &9.6247s     &3           &1.1571e-03      &3.5654       &9.8243s      &3       &2.4233e-03       &3.3533       &10.2416s     &3      \\
 $2^{4}$     &3.8127e-05     &3.8648     &45.3002s    &3           &9.4740e-05      &3.6104       &45.1551s    &3       &2.3380e-04       &3.3737       &45.6625s     &3        \\\hline
\end{tabular}
}
\label{TT05}
\end{table}

\begin{table}[!th]
\vspace{-2mm}
\renewcommand{\captionfont}{\footnotesize}
\centering
\caption{FCGS to solve Crank-Nicolson  scheme in \eqref{2.20} with $\tau=h_x=(b-a)/M_x$, $M_y=M_x.$
The exact solution is  $u(x,y,t)=e^t\left( x^4-x^3+x^2+1\right)\left(y^4-2y^3+y^2+1\right)$ }
\resizebox{\textwidth}{!}{
\begin{tabular}{ccccc|cccc|cccc}\hline\noalign{\smallskip}
\multirow{2}*{$M_x$}&  \multicolumn{4}{c}{$\gamma=0.2$}      &\multicolumn{4}{c}{$\gamma=0.5$}          &\multicolumn{4}{c}{$\gamma=0.8$}\\
                          \cline{2-13}\noalign{\smallskip}
                 &Error             &Rate          &CPU         &Iter       &Error           &Rate       &CPU        &Iter       &Error           &Rate       &CPU        &Iter \\\hline
  $2^3$      &3.7979e-03    &~~~~       &0.1996s    &4         &4.0249e-03  &~~~~    &0.2377s    &5        &4.3137e-03  &~~~~    &0.2669s    &6   \\
  $2^4$      &9.7724e-04    &1.9584       &0.9261s    &4          &1.0274e-03  &1.9699    &0.7341s    &4         &1.0901e-03  &1.9845    &1.0143s    &5    \\
  $2^5$      &2.4792e-04    &1.9788       &3.6859s    &4         &2.5942e-04  &1.9857    &3.8594s    &4         &2.7337e-04  &1.9955    &3.8412s     &4      \\
  $2^6$      &6.2440e-05    &1.9893       &12.8309s    &3         &6.5160e-05  &1.9932   &16.1971s   &4         &6.8373e-05  &1.9993    &16.0575    &4          \\\hline
\end{tabular}
}
\label{TT06}
\end{table}

Tables  \ref{TT05} and \ref{TT06}  show that  Crank-Niclson  scheme in \eqref{2.20} has the global  convergence rate
 $\mathcal {O}\left(\tau^2+h_x^{4-\gamma}+h_y^{4-\gamma}\right)$ and the computational cost is almost
 $\mathcal{O}\left(M_xM_y\log(M_yM_x)\right)$.

\section{Conclusion}
In this work, the  nonsymmetric indefinite  systems including rectangular matrices are arising from two-dimensional time-dependent nonlocal   problems.
For one-dimensional steady state nonlocal problems of \eqref{1.1}, a sharp error estimates  has been proved in \cite{CQSW:19}, but it is not
 easy to be extended to multidimensional cases.
This paper  provides  rigorous theoretical analysis  for two-dimensional steady state nonlocal problems  with multiplicative Cauchy kernel and a few technical analysis for additive Cauchy kernel. Moreover,   it reveals the supconvergence results for  time-dependent  nonlocal problems of \eqref{1.1} including two-dimensional cases.
In further, we develop the FCGS to solve the two  different algebraic systems: Kronecker product  and  block-Toeplitz Toeplitz-block-like algebraic system.
We remark that the error estimates in \cite{DHZZ:18} and \cite{TWW:13} can be obtained by following the idea given in this paper.
\begin{appendix}
\section*{Appendix A}
The matrix  ${\mathcal{G}}$ in \eqref{2.17} consists of four block-structured matrices with Toeplitz-like blocks.
Here the block-Toeplitz properties of  ${\mathcal{Q}}_{\left(M_x-1\right)\times M_x }$, ${\mathcal{P}}_{ M_x \times\left(M_x-1\right)}$ and ${\mathcal{N}}_{ M_x \times  M_x }$ are expressed  following:
\begin{equation*}
{\mathcal{Q}}=\left(\begin{array}{ccccccc}
{\mathcal{Q}}_{1,\frac{1}{2}}         & {\mathcal{Q}}_{1,\frac{3}{2}}        &{\mathcal{Q}}_{1,\frac{5}{2}}        &{\mathcal{Q}}_{1,\frac{7}{2}} & \cdots &{\mathcal{Q}}_{1,M_x-\frac{3}{2}} &{\mathcal{Q}}_{1,M_x-\frac{1}{2}}\\[2mm]
{\mathcal{Q}}_{2,\frac{1}{2}}         & {\mathcal{Q}}_{1,\frac{1}{2}}        & {\mathcal{Q}}_{1,\frac{3}{2}}       &{\mathcal{Q}}_{1,\frac{5}{2}} & \ddots &\ddots &{\mathcal{Q}}_{1,M_x-\frac{3}{2}}\\[2mm]
{\mathcal{Q}}_{3,\frac{1}{2}}         & {\mathcal{Q}}_{2,\frac{1}{2}}        & {\mathcal{Q}}_{1,\frac{1}{2}}       & {\mathcal{Q}}_{1,\frac{3}{2}}& \ddots &\ddots &\vdots\\[2mm]
\vdots                                      & \ddots                                      & \ddots                                   & \ddots                              & \ddots &\ddots                                     &{\mathcal{Q}}_{1,\frac{7}{2}}\\[2mm]
{\mathcal{Q}}_{M_x-2,\frac{1}{2}}   & \ddots                                     & \ddots                                    & \ddots                             & \ddots   &{\mathcal{Q}}_{1,\frac{3}{2}}        &{\mathcal{Q}}_{1,\frac{5}{2}}\\[2mm]
{\mathcal{Q}}_{M_x-1,\frac{1}{2}}   & {\mathcal{Q}}_{M_x-2,\frac{1}{2}} & \cdots                                    & {\mathcal{Q}}_{3,\frac{1}{2}} & {\mathcal{Q}}_{2,\frac{1}{2}}          &{\mathcal{Q}}_{1,\frac{1}{2}}        &{\mathcal{Q}}_{1,\frac{3}{2}}
\end{array}\right)_{\left(M_x-1\right)\times M_x };
\end{equation*}
and
 \begin{equation*}
{\mathcal{P}}=\left(\begin{array}{cccccc}
{\mathcal{P}}_{\frac{1}{2},1} & {\mathcal{P}}_{\frac{1}{2},2} & {\mathcal{P}}_{\frac{1}{2},3}& \cdots &{\mathcal{P}}_{\frac{1}{2},M_x-2} &{\mathcal{P}}_{\frac{1}{2},M_x-1}\\[2mm]
{\mathcal{P}}_{\frac{3}{2},1} & {\mathcal{P}}_{\frac{1}{2},1} & {\mathcal{P}}_{\frac{1}{2},2}& \ddots &\ddots &{\mathcal{P}}_{\frac{1}{2},M_x-2}\\[2mm]
{\mathcal{P}}_{\frac{5}{2},1} & {\mathcal{P}}_{\frac{3}{2},1} & {\mathcal{P}}_{\frac{1}{2},1}& \ddots &\ddots &\vdots\\[2mm]
{\mathcal{P}}_{\frac{7}{2},1} & {\mathcal{P}}_{\frac{5}{2},1} & {\mathcal{P}}_{\frac{3}{2},1}& \ddots &\ddots &{\mathcal{P}}_{\frac{1}{2},3}\\[2mm]
\vdots                              & \ddots                               & \ddots                             &\ddots &\ddots &{\mathcal{P}}_{\frac{1}{2},2}\\[2mm]
{\mathcal{P}}_{M_x-\frac{3}{2},1} & \ddots                        & \ddots                            & \ddots  &{\mathcal{P}}_{\frac{3}{2},1} &{\mathcal{P}}_{\frac{1}{2},1}\\[2mm]
{\mathcal{P}}_{M_x-\frac{1}{2},1} &{\mathcal{P}}_{M_x-\frac{3}{2},1} & \cdots &{\mathcal{P}}_{\frac{7}{2},1} &{\mathcal{P}}_{\frac{5}{2},1} &{\mathcal{P}}_{\frac{3}{2},1}
\end{array}\right)_{ M_x \times\left(M_x-1\right)};
\end{equation*}
and
 \begin{equation*}
{\mathcal{N}}=\left(\begin{array}{ccccc}
{\mathcal{N}}_{\frac{1}{2},\frac{1}{2}} & {\mathcal{N}}_{\frac{1}{2},\frac{3}{2}} & {\mathcal{N}}_{\frac{1}{2},\frac{5}{2}}& \cdots &{\mathcal{N}}_{\frac{1}{2},M_x-\frac{1}{2}}\\[2mm]
{\mathcal{N}}_{\frac{3}{2},\frac{1}{2}} & {\mathcal{N}}_{\frac{1}{2},\frac{1}{2}} & {\mathcal{N}}_{\frac{1}{2},\frac{3}{2}}& \ddots &\vdots\\[2mm]
{\mathcal{N}}_{\frac{5}{2},\frac{1}{2}} & {\mathcal{N}}_{\frac{3}{2},\frac{1}{2}} & {\mathcal{N}}_{\frac{1}{2},\frac{1}{2}}& \ddots &{\mathcal{N}}_{\frac{1}{2},\frac{5}{2}}\\[2mm]
\vdots                                            & \ddots                                             & \ddots                                           &\ddots   &{\mathcal{N}}_{\frac{1}{2},\frac{3}{2}}\\[2mm]
{\mathcal{N}}_{M_x-\frac{1}{2},\frac{1}{2}} & \cdots                                       & {\mathcal{N}}_{\frac{5}{2},\frac{1}{2}}& {\mathcal{N}}_{\frac{3}{2},\frac{1}{2}} &{\mathcal{N}}_{\frac{1}{2},\frac{1}{2}}
\end{array}\right)_{ M_x \times M_x}.
\end{equation*}
It should be noted that the above matrices ${\mathcal{Q}}$ and ${\mathcal{P}}$ have the similar structure properties with $\mathcal{Q}$ and $\mathcal{P}$ in \eqref{2.3}.

\section*{Appendix B}
\begin{algorithm}[!htp]
\caption{Fast Conjugate Gradient Squared for \eqref{2.3} }
\label{cgs1}
\begin{algorithmic}[1]
\STATE   Residual
$\left[\begin{array}{l} r_w^0 \\ r_v^0 \end{array}\right]=\left[\begin{array}{l} F_w  \\ F_v  \end{array}\right]-
\left[\begin{array}{cc}\mathcal {D}_1-\mathcal{M} & -\mathcal{Q} \\ -\mathcal{P} & \mathcal {D}_2-\mathcal{N}\end{array}\right]\left[\begin{array}{l} w \\ v \end{array}\right]$;
$w=\bf 0$, $v=\bf 0$.\\[1.5mm]
\STATE $ r_w^*=r_w^0$, $ r_v^*=r_v^0$; Set $p_w:=z_w:=r_w^0, p_v:=z_v:=r_v^0$;\\[1.5mm]
\STATE {\bf While} residual$>$tolerance $\&$ $j<{\rm maxit}$ {\bf do}:\\[1.5mm]
\STATE \qquad $\left[\begin{array}{l} \mathcal{A}p_w \\ \mathcal{A}p_v \end{array}\right]=
\left[\begin{array}{cc}\mathcal {D}_1-\mathcal{M} & -\mathcal{Q} \\ -\mathcal{P} & \mathcal {D}_2-\mathcal{N}\end{array}\right]\left[\begin{array}{l} p_w \\ p_v \end{array}\right]$\\[1.5mm]
\STATE \qquad $\alpha=\frac{\left([r_w^j;r_v^j],[r_w^*;r_v^*]\right)}{\left([\mathcal{A}p_w;\mathcal{A}p_v],[r_w^*;r_v^*]\right)}$\\[1.5mm]
\STATE \qquad $q_w=z_w-\alpha\mathcal{A}p_w$, $q_v=z_v-\alpha\mathcal{A}p_v$\\[1.5mm]
\STATE \qquad $w=w+\alpha\left(z_w+q_w\right)$,~~$v=v+\alpha\left(z_v+q_v\right)$\\[1.5mm]
\STATE \qquad
 Residual $\left[\begin{array}{l} r_w^{j+1} \\ r_v^{j+1} \end{array}\right]=\left[\begin{array}{l} r_w^j \\ r_v^j \end{array}\right]-
\left[\begin{array}{cc}\mathcal {D}_1-\mathcal{M} & -\mathcal{Q} \\ -\mathcal{P} & \mathcal {D}_2-\mathcal{N}\end{array}\right]\left[\begin{array}{l} z_w+q_w \\ z_v+q_v \end{array}\right]$\\[1.5mm]
\STATE  \qquad  $\beta =\frac{\left([r_w^{j+1};r_v^{j+1}],[r_w^*;r_v^*]\right)}{\left([r_w^j;r_v^j],[r_w^*;r_v^*]\right)}$\\[1.5mm]
\STATE \qquad $z_w=r_w^{j+1}+\beta q_w$, $z_v=r_v^{j+1}+\beta q_v$\\[1.5mm]
\STATE \qquad  $p_w=z_w+\beta \left(q_w+\beta p_w\right)$, $p_v=z_v+\beta \left(q_v+\beta p_v\right)$\\[1.5mm]
\STATE  {\bf Endwhile}\\[1.5mm]
\STATE  Return $w$, $v$
\end{algorithmic}
\end{algorithm}
\begin{algorithm}[!h]
\caption{Fast Conjugate Gradient Squared for time-dependent problems}
\label {cgs2}
\begin{algorithmic}[1]
\STATE $t:=0$
\STATE {\bf While} {$t<T$} {\bf Do}:
\STATE $t:=t+\tau$
\STATE solve  time-dependent problems   by Algorithm \ref{cgs1} for 1D or Algorithm \ref{cgs3} for 2D
\STATE {\bf EndWhile}
\end{algorithmic}
\end{algorithm}
\begin{algorithm}[!htp]
\caption{Fast Fourier transform algorithm for $\left(\mathcal{G}_x\otimes \mathcal{G}_y\right){U}$ }
\label{mat-vec2}
\begin{algorithmic}[1]
\STATE  {\bf For} $i=1,2,\cdots, M_x-1$ {\bf Do}:
\STATE  ~~~~Compute $V_i=\mathcal{G}_y U^T_{i}$ by FFT;
\STATE  {\bf EndDo}
\STATE  {\bf For} $i=M_x,M_x+1,\cdots, 2M_x-1$ {\bf Do}:
\STATE  ~~~~Compute $V_i=\mathcal{G}_y U^T_{\frac{i}{2}}$ by FFT;
\STATE  {\bf EndDo}
\STATE   Give the notation with row vectors of $V$, i.e.,\\
~~~~~$V=\left(\begin{array}{c}V_{1}\\V_{2}\\\vdots\\V_{j}\\\vdots\\V_{2M_y-1}\end{array}\right)$
with $V_{j}=\left(V_{j ,1},V_{j ,2},\cdots,V_{j ,2M_x-1}\right)$;
\STATE  {\bf For} $j=1,2,\cdots, 2M_y-1$ {\bf Do}:
\STATE  ~~~~Compute $W_j=\mathcal{G}_x V_{j}^T$ by FFT;
\STATE  {\bf EndDo}
\STATE  Give the notation with row vectors of $W$, i.e., \\
~~~~~$W=\left(\begin{array}{c}W_1\\W_2\\\vdots \\W_{i} \\\vdots \\W_{2M_x-1}\end{array}\right)$ with $W_{i}=\left(W_{i,1},W_{i,2},\cdots,W_{i,2M_y-1}\right)$;
\STATE  Return $\left(\mathcal{G}_x\otimes \mathcal{G}_y\right){U}=\left(W_1,W_2,\cdots ,W_{i},\cdots,W_{2M_x-1}\right)^T$.
\end{algorithmic}
\end{algorithm}

\begin{algorithm}[!h]
\caption{Fast Conjugate Gradient Squared for 2D steady equation}
\label{cgs3}
\begin{algorithmic}[1]
\STATE   Initialize:${U}=\bf 0$; ${R}^0= {F}-{{\mathcal{A}}}{U}$; $ {R}^*={R}^0$.
\STATE Set ${p}:={z}:={R}^0$;
\STATE {\bf While} residual$>$tolerance $\&$ $j<{\rm maxit}$ {\bf Do}:
\STATE \qquad ${{\mathcal{A}}}_p={{\mathcal{A}}}{p}$, Compute ${{\mathcal{A}}}{p}$ by referring to the Algorithm \ref{mat-vec2}
\STATE \qquad $\alpha=\frac{\left({R}^j,{R}^*\right)}{\left({{\mathcal{A}}}_p,{R}^*\right)}$
\STATE \qquad ${q}={z}-\alpha{{\mathcal{A}}}_p$
\STATE \qquad ${U}={U}+\alpha\left({z}+{q}\right)$
\STATE \qquad ${R}^{j+1}={R}^j -{{\mathcal{A}}}\left({z}+{q}\right)$, compute ${{\mathcal{A}}}\left({z}+{q}\right)$ by referring to the Algorithm \ref{mat-vec2}
\STATE  \qquad  $\beta =\frac{\left({R}^{j+1},{R}^*\right)}{\left({R}^j,{R}^*\right)}$
\STATE \qquad ${z}={R}^{j+1}+\beta {q}$
\STATE \qquad  ${p}={z}+\beta \left({q}+\beta {p}\right)$
\STATE  {\bf EndDo}
\STATE  Return ${U}$
\end{algorithmic}
\end{algorithm}
\end{appendix}

\end{document}